\documentclass[11pt]{amsart}
\usepackage{amsfonts,latexsym,amsthm,amssymb,amsmath,amscd,euscript,tikz, tikz-cd}
\usepackage[alphabetic, msc-links, bibtex-style, nobysame]{amsrefs}
\usepackage{stackengine}
\usepackage{framed}
\usepackage{xfrac}
\usepackage[makeroom]{cancel}
\usepackage{faktor}
\usepackage{braket}
\usepackage{pgf,tikz,pgfplots}
\usepgfplotslibrary{groupplots}
\pgfplotsset{compat=1.18}
\usepgfplotslibrary{fillbetween} 
\usepackage{mathrsfs}
\usetikzlibrary{arrows}
\definecolor{wrwrwr}{rgb}{0.3803921568627451,0.3803921568627451,0.3803921568627451}
\definecolor{rvwvcq}{rgb}{0.08235294117647059,0.396078431372549,0.7529411764705882}
\definecolor{mblue}{rgb}{0.2, 0.3, 0.8}
\definecolor{morange}{rgb}{1, 0.5, 0}
\definecolor{mgreen}{rgb}{0.1, 0.4, 0.2}
\definecolor{mred}{rgb}{0.5, 0, 0}
\definecolor{ForestGreen}{RGB}{34,139,34}
\usepackage{float}

\numberwithin{equation}{section}

\usepackage{lmodern}
\usepackage{alphabeta}
\usepackage{supertabular}
\usepackage{amssymb}
\usepackage{enumerate}
\usepackage{stmaryrd}
\usepackage{bbm}
\usepackage{mathtools}
\usepackage{setspace}
\usepackage{tikz,tikz-cd}
\usetikzlibrary{matrix,calc,positioning,arrows,decorations.pathreplacing,patterns,knots}

\newcommand{\la}{\langle}
\newcommand{\rg}{\rangle}

\newtheorem{theorem}{{Theorem}}[section]
\newtheorem*{theorem*}{Theorem}
\newtheorem{lemma}[theorem]{Lemma}
\newtheorem{proposition}[theorem]{Proposition}

\newtheorem*{corollary*}{Corollary}

\theoremstyle{definition}

\newtheorem{remark}{Remark}

\usepackage{lmodern,url,enumerate,mathtools}
\usepackage[hmargin = 0.95in,vmargin=0.95in]{geometry}
\usepackage{graphicx}
\usepackage{subcaption}

\usepackage{hyperref}
    \hypersetup{colorlinks=true,citecolor=ForestGreen,linkcolor = blue,urlcolor =black,linkbordercolor={1 0 0}}

\newcommand{\ve}{\varepsilon}

\newcommand{\mr}[1]{{\rm #1}}
\newcommand{\mres}{\mathbin{\vrule height 1.6ex depth 0pt width
0.13ex\vrule height 0.13ex depth 0pt width 1.3ex}}

\newcommand{\cA}{\mathcal{A}}\newcommand{\cB}{\mathcal{B}}

\newcommand{\cH}{\mathcal{H}}

\newcommand{\cK}{\mathcal{K}}\newcommand{\cL}{\mathcal{L}}
\newcommand{\cM}{\mathcal{M}}\newcommand{\cN}{\mathcal{N}}

\newcommand{\cR}{\mathcal{R}}

\newcommand{\cU}{\mathcal{U}}

\newcommand{\bE}{\mathbb{E}}
\newcommand{\bG}{\mathbb{G}}

\newcommand{\bR}{\mathbb{R}}
\newcommand{\bS}{\mathbb{S}}

\newcommand{\bZ}{\mathbb{Z}}

\newcommand{\nc}{\newcommand}
\nc{\sn}{\mr{sn}}
\nc{\cn}{\mr{cn}}
\nc{\dn}{\mr{dn}}

\nc{\on}{\operatorname}
\nc{\p}{\partial}
\nc{\ol}{\overline}
\nc{\ul}{\underline}
\nc{\pa}{\partial}

\nc{\pb}{\partial_b}
\nc{\pc}{\partial_c}
\nc{\pd}{\partial_d}
\nc{\pe}{\partial_e}
\nc{\pf}{\partial_f}
\nc{\pg}{\partial_g}
\nc{\ph}{\partial_h}
\nc{\pari}{\partial_i}
\nc{\pj}{\partial_j}
\nc{\pk}{\partial_k}
\nc{\pl}{\partial_l}
\nc{\pell}{\partial_\ell}
\nc{\parm}{\partial_m}
\nc{\pn}{\partial_n}
\nc{\po}{\partial_o}
\nc{\pp}{\partial_p}
\nc{\pq}{\partial_q}
\nc{\pr}{\partial_r}
\nc{\ps}{\partial_s}
\nc{\pt}{\partial_t}
\nc{\pu}{\partial_u}
\nc{\pv}{\partial_v}
\nc{\pw}{\partial_w}
\nc{\px}{\partial_x}
\nc{\py}{\partial_y}
\nc{\pz}{\partial_z}

\numberwithin{equation}{section}
\makeatletter
\@addtoreset{equation}{section}
\makeatother

\usepackage{mathtools}
\mathtoolsset{showonlyrefs}

\author[Benjy Firester]{Benjy Firester}
\address{{ \href{mailto:benjyfir@mit.edu}{benjyfir@mit.edu}} \hfill Department of Mathematics, MIT}
\author[Raphael Tsiamis]{Raphael Tsiamis}
\address{
{\href{mailto:r.tsiamis@columbia.edu}{r.tsiamis@columbia.edu}} \hfill Department of Mathematics, Columbia University}

\title[The anisotropic Michael-Simon inequality for surfaces in every codimension]{The anisotropic Michael-Simon inequality for surfaces in every codimension} %
\date{\today}

\begin{document}

\begin{abstract}
We prove a Michael-Simon inequality for $2$-varifolds in $\mathbb{R}^N$, in arbitrary codimension, for anisotropies sufficiently close to the area functional. Building on the ideas of Almgren and De Philippis-Pigati, we introduce a projection method for anisotropic stress measures yielding geometric inequalities that apply in arbitrary dimension and codimension.
\end{abstract}

\maketitle

\vspace{-0.2in}

\section{Introduction}

The monotonicity formula for the area functional is a fundamental tool in the study of minimal submanifolds and their measure-theoretic counterparts, stationary varifolds~\cite{colding-minicozzi}.
Its many consequences include density bounds, compactness properties, and deep regularity results proved in the landmark works of Allard~\cites{allard-first-variation , allard-regularity }.
In the absence of a monotonicity formula for anisotropic integrands, a Michael-Simon inequality~\cite{michael-simon} is the primary tool for establishing such properties; this problem is posed as Question 8.2 in~\cite{gmt-differential-inclusions}.
Recently, De Philippis and Pigati~\cite{anisotropic-michael-simon} utilized ideas of Almgren to prove a Michael-Simon inequality for surfaces in $\bR^3$, with respect to anisotropies sufficiently close to the area functional.
They also showed, for integrands satisfying the atomic condition, that the Michael-Simon inequality is equivalent to compactness of rectifiable varifolds.

We prove a Michael-Simon inequality for $2$-varifolds in $\bR^N$, extending the $N=3$ result of~\cite{anisotropic-michael-simon}.

\begin{theorem}\label{thm:anisotropic-MS}
For every $N \geq 3$, there exist constants $\ve_N>0$ and $C_N<\infty$ with the following property.
For every $\Psi \in C^1(\bG(N,2) , (0,\infty) )$ with $\sup_{T \in \bG(N,2)} \|B_{\Psi}(T) - T \| \leq \ve_N$, every rectifiable $2$-varifold $V = v(M,\theta)$ in $\bR^N$ with finite total mass and finite first variation with respect to the anisotropy $\Psi$ satisfies
\[
\|V\| (\bR^N) \leq C_N \, \cH^2(M)^{\frac{1}{2}} \, |\delta_{\Psi} V|(\bR^N).
\]
In particular, if $\Theta^2(\|V\|,x) \geq \theta_0 > 0$ holds $\|V\|$-a.e.~$x$, then 
\[
\|V\|(\bR^N)^{\frac{1}{2}} \leq C_N \theta_0^{- \frac{1}{2}} |\delta_{\Psi} V|( \bR^N).
\]
\end{theorem}

\begin{remark}
While this manuscript was in the final stages of writing, the remarkable work of Gennaioli and Rindler~\cite{vanishing-mass-conjecture} appeared on arXiv, resolving Bouchitt\'e's vanishing mass conjecture.
We realized that some important compensated compactness results for measures proved in their work, notably~\cite{vanishing-mass-conjecture}*{Propositions 3.1 and 3.5}, could be combined with the projection techniques we develop in Section~\ref{section:smallest-eigenvalue} to establish the anisotropic Michael-Simon inequality for varifolds of every dimension and codimension.
We carry out this alternative approach in~\cite{near-area-functional}.
\end{remark}

Theorem~\ref{thm:anisotropic-MS} has several important consequences for rectifiable $2$-varifolds, notably density bounds, rectifiability of limits, compactness of varifolds, and a strong constancy theorem, we refer the reader to~\cites{derosa-ghiraldin , derosa-kolasinski , derosa-tione-regularity } developing the properties of anisotropic varifolds, and highlight the introduction of~\cite{anisotropic-michael-simon} where some of these properties are discussed.
We expand on these consequences for more general anisotropic varifolds in~\cite{near-area-functional}.

\subsection{Strategy of the proof}

The proof of Theorem~\ref{thm:anisotropic-MS} utilizes many techniques introduced in the beautiful recent work of De Philippis and Pigati~\cite{anisotropic-michael-simon }.
The key novelty is a projection method for anisotropic stress measures, introduced in Section~\ref{section:smallest-eigenvalue} and inspired by Brendle's ingenious proof of the sharp isoperimetric inequality using optimal transportation~\cites{brendle-sharp-isoperimetric , brendle-eichmair }.
In Brendle's argument, optimal transportation is combined with the normal bundle to lift the problem onto the ambient space, equipped with a rotationally symmetric density.
In our setting, the analogous role is played by the normalized Haar measure on the Grassmannian $\bG(N,m)$.
The projection method also allows us to reduce the Michael-Simon inequality to a quantitative non-concentration result for matrix-valued measures on $\bR^m$, which forms a key step in the proof of~\cite{near-area-functional}.

After these reductions, the problem is reduced to an inequality for a functional on vector-valued measures.
In Section~\ref{section:smallest-eigenvalue}, we construct such a functional, denoted by $\lambda_{2,\beta}(A) = \lambda_2(A) - \frac{4-\pi}{8} \textup{tr} \, A$, where $\lambda_2$ denotes the smallest eigenvalue of a $2\times 2$ matrix.
Crucially, Lemma~\ref{lemma:key-m=2} shows that $\lambda_{2,\beta}$ has positive Haar average under conjugation by elements of $\bG(N,2)$, for every $N \geq 3$.
On the other hand, Proposition~\ref{prop:A-div-S} provides a Michael-Simon type upper bound for the application of $\lambda_{2,\beta}$ on matrix-valued measures.
This estimate is a sharp refinement of the nonlinear inequality for planar vector fields from~\cite{anisotropic-michael-simon}*{Theorem 3.1}, which we obtain by controlling terms carefully in preliminary lemmas of Sections~\ref{section:elementary-lemmas} and~\ref{section:smallest-eigenvalue}.
We refer the reader to the discussion following Theorem~\ref{thm:kakeya-type-inequality} for more details on these steps.
In Section~\ref{section:proof-surfaces}, we combine these ingredients to prove Theorem~\ref{thm:anisotropic-MS}.

\smallskip \noindent \textbf{Acknowledgments.}
We are very grateful to Antonio de Rosa for many helpful discussions on anisotropic problems.
We also thank Simon Brendle and Toby Colding for inspiring conversations.
BF recognizes support from a Simons Dissertation Fellowship, a MathWorks Fellowship, and the Citadel GQS PhD Fellowship.
RT is supported by a Simons Dissertation Fellowship, an A.G.~Leventis Foundation Scholarship, and an Onassis Foundation Scholarship.

\section{Preliminaries}

Let $\Psi: \bG(N,m) \to (0,+\infty)$ be $C^1$.
We denote by $A^t$ the transpose of an endomorphism $A$ and identify $\bG(N,m)$ with the set of rank-$m$ orthogonal projections in $\bR^N$, namely 
\[
\bG(N,m) = \{ T \in \bR^{N \times N} : T = T^t, \; T^2 = T, \; \on{tr} T = m \}.
\]
Given an $m$-plane $T \in \bG(N,m)$, we define a tensor field $B_{\Psi}(T)$ by
\[
B_{\Psi}(T) = \Psi(T) T + T^{\perp} \, d \Psi(T) \, T, \qquad \text{where } \; d \Psi(T) := D \Psi(T) + D \Psi(T)^t.
\]
By this definition, $B_{\Psi}(T)$ satisfies $B_{\Psi}(T) T = B_{\Psi}(T)$ and
\begin{equation}\label{eqn:operator-norm}
    \sup_{T \in \bG(N,m)} \| B_{\Psi}(T) - T \| \leq \| \Psi - 1 \|_{C^0} + 2 \, \|D \Psi\|_{C^0} \leq 2 \, \| \Psi - 1 \|_{C^1}
\end{equation}
in operator norm.
Indeed, $\| T^{\perp} \, d \Psi(T) \, T \| \leq 2 \, \|D\Psi(T)\|$ because $T,T^{\perp}$ are orthogonal projections.
In particular, $\sup_{T \in \bG(N,m)} \|B_{\Psi}(T) \| \leq C_N$ whenever $\sup_{T \in \bG(N,m)} \|B_{\Psi}(T) - T \| \leq c_N$.

Let $V = v(M,\theta)$ be a rectifiable $m$-varifold in an open set $\Omega \subset \bR^N$, namely a Radon measure on $\Omega \times \bG(N,m)$.
We write $\|V\|$ for the weight measure, which satisfies $\|V\|(A) := V(A \times \bG(N,m))$ for every Borel set $A \subseteq \Omega$.
If $\Theta^m(\|V\|,x_0) \geq \theta_0$ a.e., then we can choose the rectifiable set $M$ to satisfy $\theta \geq \theta_0$ $\cH^m$-a.e., and hence $\cH^m(M) \leq \theta_0^{-1} \|V\|(\Omega)$.

The \textbf{anisotropic first variation} of $V$ with respect to the flow generated by a Lipschitz vector field $g \in \textup{Lip}_c(\Omega;\bR^N)$ is computed in~\cite{derosa-ghiraldin}*{Lemma A.2} as
\begin{equation}\label{eqn:first-variation}
[\delta_{\Psi} V] (g) = \int \la B_{\Psi}(T), D g(x) \rg \, dV(x,T) =
\int_{M} \la B_{\Psi} (T_x M), Dg(x) \rg \, \theta \, d \cH^m(x).
\end{equation}

\subsection{Some elementary lemmas}\label{section:elementary-lemmas}

We record some integral bounds for single-variable functions.
As in~\cite{anisotropic-michael-simon}*{Lemma 3.11}, we need a version of Riesz's rising sun lemma.
\begin{lemma}\label{lemma:rising-sum}
For every absolutely continuous function $\xi \in W^{1,1}([a,b])$, it holds that
\[
    \int_{ [a,b] \setminus E } ( \xi')_+ \leq \int_a^b (\xi')_-\, , \qquad \text{where } \; E := \{ s \in [a,b] : \xi(t) > \xi(s) \; \textup{for every} \; t>s \}.
\]
\end{lemma}
\begin{proof}
    If $\xi(b) \leq \xi(a)$, then $\int_a^b (\xi')_+ \leq \int_a^b (\xi')_-$ proves the claim.
    Suppose now that $\xi(b) > \xi(a)$, so for $\lambda \in [\xi(a), \xi(b)]$, the last point in the level set $\xi^{-1}(\lambda)$ belongs to $E$.
    Hence, up to a null set, the image $\xi(E)$ contains the interval $[\xi(a), \xi(b)]$, so the one-dimensional area formula implies $\int_E (\xi')_+ \geq \xi(b) - \xi(a) = \int_a^b (\xi')_+ - (\xi')_-$.
    The desired inequality follows by rearranging terms.
\end{proof}

We also prove an integral bound for the slopes of curves contained in a slab of $\bR^2$, which is similar to~\cite{anisotropic-michael-simon}*{Lemma 3.12}.
The following Lemma~\ref{lemma:lipschitz-curve} has the advantage that the constants in the bound are explicit and equal to $1$, which will be an essential ingredient to proving the desired Michael-Simon inequality for surfaces in arbitrary codimension.

\begin{lemma}\label{lemma:lipschitz-curve}
Given a unit direction $e \in \bS^{m-1}$, we consider the strip
\[
\cU^e_j := \{ x \in \bR^m : j \leq \la x, e \rg \leq j+1 \}.
\]
Let $\gamma : [a,b] \to \cU_j^e$ be a Lipschitz curve and write $\dot{\gamma}(t)^{\perp} := \dot{\gamma}(t) - \la \dot{\gamma}(t), e \rg e$.
Then it holds that
\[
\int_a^b ( \la \dot{\gamma}(t), e \rg - |\dot{\gamma}(t)^{\perp}| )_+ \, dt \leq 1 + \int_a^b ( \la \dot{\gamma}(t), e \rg - |\dot{\gamma}(t)^{\perp}|)_- \, dt.
\]
Moreover, if $\la \gamma(b) , e \rg \leq \la \gamma(a), e \rg$, and in particular if the curve is closed, we have
\[
\int_a^b ( \la \dot{\gamma}(t), e \rg - |\dot{\gamma}(t)^{\perp}| )_+ \, dt \leq \int_a^b ( \la \dot{\gamma}(t), e \rg - |\dot{\gamma}(t)^{\perp}|)_- \, dt.
\]
Finally, given any non-negative function $h \in W^{1,1}([a,b])$ with $\min h = 0$, we have
\[
\int_a^b h(t) \, ( \la \dot{\gamma}(t), e \rg - |\dot{\gamma}(t)^{\perp}|)_+ \, dt \leq \int_a^b |h'| \, dt + \int_a^b h(t) \, ( \la \dot{\gamma}(t), e \rg - |\dot{\gamma}(t)^{\perp}|)_- \, dt.
\]
\end{lemma}
\begin{proof}
    Because $\gamma$ is Lipschitz, its components are absolutely continuous and differentiable a.e.
    We set $f(t) := \la \dot{\gamma}(t), e \rg - |\dot{\gamma}(t)^{\perp}|$ and use $f = f_+ - f_-$ and the absolute continuity of $\gamma$ to compute
    \[
    \int_a^b f_+(t) \, dt = \int_a^b f_-(t) \, dt + \int_a^b \la \dot{\gamma}(t), e \rg \, dt - \int_a^b |\dot{\gamma}(t)^{\perp}| \, dt \leq ( \la \gamma(b), e \rg - \la \gamma(a), e \rg)_+ + \int_a^b f_-(t) \, dt.
    \]
    Moreover, $ ( \la \gamma(b), e \rg - \la \gamma(a), e \rg)_+  = 0$ for a curve with $\la \gamma(b), e \rg \leq \la \gamma(a), e \rg$, including closed curves.
    Also, $ ( \la \gamma(b), e \rg - \la \gamma(a), e \rg)_+ \leq 1$ for a general curve with $\gamma([a,b]) \subseteq \cU^e_j$.
    This proves our claim.

    Regarding the second inequality, the one-dimensional coarea formula implies that
    \[
    \int_a^b h(t) f_+(t) \, dt = \int_0^{\max h} \int_{ \{ h > s \} } f_+(t) \, dt \, ds.
    \]
    For a.e.~$s$, the region $\{ h> s\}$ is an at most countable union of disjoint open intervals.
    By the one-dimensional area formula,
    \begin{equation}\label{eqn:no-of-components}
    \int_0^{\max h} \# \{ h=s \} \, ds = \int_a^b |h'| \, dt < \infty,
    \end{equation}
    so for a.e.~$s$, we have $\# \{h=s\} < \infty$, and consequently only finitely many components $N_s$ occur for those $s$.
    Applying the previous inequality to each component shows that
    \[
    \int_{ \{ h>s \} } f_+(t) \, dt  \leq N_s + \int_{  \{ h>s \} } f_-(t) \, dt.
    \]
    Because $\min h = 0$, every component of $\{ h> s\}$ has at least one boundary point in $\{ h=s\}$, and hence $N_s \leq \# \{ h=s\}$.
    The claimed inequality follows from combining~\eqref{eqn:no-of-components} with these steps.
\end{proof}

\subsection{Tools from measure theory}\label{section:measure-theory}

The proof of Theorem~\ref{thm:anisotropic-MS} will invoke some key tools from measure theory, notably Smirnov's decomposition for currents and the Reshetnyak continuity theorem.
We first record a standard strict approximation lemma for vector-valued measures.
\begin{lemma}\label{lemma:strict-approximation}
    For $n, d \geq 1$, consider a measure $\mu \in \cM( \bR^n; \bR^d)$ and let $\rho_{\ve}$ be a standard non-negative mollifier, with $\mu_{\ve} := \rho_{\ve} \ast \mu \in \cL^n$.
    Then, $\mu_{\ve} \to \mu$ strictly as $\ve \downarrow 0$, meaning that
    \[
    \mu_{\ve} \xrightharpoonup{*} \mu, \qquad |\mu_{\ve}|(\bR^n) \to |\mu|(\bR^n).
    \]
    More generally, let $\cA = \sum_{j=1}^n A_j \partial_j$ be a first-order linear operator and suppose that $\cA \mu \in \cM ( \bR^n; \bR^{\ell})$.
    Then, we have $\cA \mu_{\ve} = ( \rho_{\ve} \ast \cA \mu) \, \cL^n$ and $\cA \mu_{\ve} \to \cA \mu$ strictly.
    If the measures $\mu$ and $\cA \mu$ are compactly supported, both families of variation measures are uniformly tight.
\end{lemma}
\begin{proof}
    The weak-* convergence is standard from the fact that $\{ \rho_{\ve} \}$ is an approximate identity.
    By Jensen's inequality and Fubini's theorem, we have
    \[
    |\mu_{\ve}|(\bR^n) = \int_{\bR^n} |\rho_{\ve} \ast \mu |\, dx \leq |\mu| (\bR^n).
    \]
    On the other hand, the weak-* lower semicontinuity of the total variation gives
    \[
    |\mu|(\bR^n) \leq \liminf_{\ve \downarrow 0} |\mu_{\ve}|(\bR^n),
    \]
    and hence $\mu_{\ve} \to \mu$ strictly; this proves the first assertion.

    For the second assertion, we use the fact that constant coefficient differential operators commute with convolution, so $\cA \mu_{\ve} = ( \rho_{\ve} \ast \cA \mu) \, \cL^n$.
    Hence, the same result applies to $\cA \mu$.
    Finally, for a compactly supported standard mollifier with $\textup{spt} \, \rho_{\ve} \subset B_{C \ve}$, we have
    \[
    \textup{spt} \, \mu_{\ve} \subset \textup{spt} \, \mu + B_{C \ve}, \qquad \textup{spt} \, \cA \mu_{\ve} \subset \textup{spt} \, \cA \mu + B_{C \ve}.
    \]
    Thus, the above measures are uniformly tight.
    This completes the proof.
\end{proof}

Under the strict convergence of measures as above, we can invoke Reshetnyak's continuity theorem (see, for example,~\cite{reshetnyak} and~\cite{ambrosio-fusco-pallara}*{Theorem 2.39}) in the following form.
\begin{lemma}\label{lemma:reshetnyak-continuity}
    Let $\mu_j, \mu \in \cM( \Omega; \bR^d)$ be vector-valued measures on an open set $\Omega \subset \bR^n$ that converge strictly, $\mu_j \to \mu$, meaning that $\mu_j \xrightharpoonup{*} \mu$ and $|\mu_j|(\Omega) \to |\mu|(\Omega)$.
    Let $F: \Omega \times \bR^d \to \bR$ be a continuous, positively $1$-homogeneous function in the second variable satisfying a uniform linear bound,
    \[
    F(x,tz) = t \, F(x,z) \quad \text{for every } \; t \geq 0, \qquad |F(x,z)| \leq C |z| \quad \text{for every } \; (x,z) \in \Omega \times \bR^d.
    \]
    Then, it holds that
    \[
    \int_{\Omega} F \Bigl( \frac{d \mu_j}{d |\mu_j|} \Bigr) \, d |\mu_j| \to \int_{\Omega} F \Bigl( \frac{d \mu}{d |\mu|} \Bigr) \, d |\mu|.
    \]
\end{lemma}
We will also use the following planar consequence of Smirnov's decomposition theorem for normal $1$-currents~\cite{smirnov}.
We identify an $\bR^2$-valued measure $J$ of finite mass and finite distributional divergence with its associated normal $1$-current; under this identification, $\mathbf{M} (\partial J) = |\textup{div} \, J|$.
\begin{theorem}\label{thm:smirnov}
Let $J$ be a normal $1$-current in $\bR^2$.
Then, there exist non-negative measures $\lambda, \mu$; a family of injective arcs $\{ \gamma_{\sigma} \}_{\sigma \in \Sigma}$ of finite length; and a family of integral $1$-cycles $\{ \Gamma_{\xi} \}_{\xi \in \Xi}$, each of which is a countable union of finite-length loops, such that
\[
J = \int_{\Sigma} [\![ \gamma_{\sigma} ]\!] \, d \lambda(\sigma) + \int_{\Xi} \Gamma_{\xi} \, d \mu(\xi)
\]
with no cancellation of mass.

Moreover, the arc part has no boundary cancellation: for every open set $\Omega \subset \bR^2$, it holds that
\[
\mathbf{M}( \partial J ; \Omega) = \int \mathbf{M} ( \partial [\![ \gamma_{\sigma} ]\!] ; \Omega) \, d \lambda(\sigma)
\]
Consequently, if $\Sigma_{\Omega}$ denotes the collection of arcs having at least one endpoint in $\Omega$, we have
\[
\Sigma_{\Omega} := \{ \sigma : \gamma_{\sigma} \; \text{has an endpoint in } \Omega \}, \qquad \lambda(\Sigma_{\Omega}) \leq \mathbf{M} ( \partial J; \Omega).
\]
\end{theorem}
The second identity implies that the cycle part has zero boundary and the arc decomposition has no boundary cancellation.
Since there is no mass cancellation, for $\lambda$-a.e.~$\sigma$, the oriented tangent to $\gamma_{\sigma}$ agrees a.e.~with the polar direction of $J$.
The cycle part also satisfies the analogous statement.

Finally, we discuss a method of assigning a smooth direction to a vector field through its zero set by truncation and extension.
The following technical lemma modifies~\cite{anisotropic-michael-simon}*{Lemma 3.15} to every dimension using a simpler argument.
\begin{lemma}\label{lemma:S-vector-field}
Let $m \geq 2$ and consider a vector field $S \in C_c^{\infty}(\bR^m;\bR^m)$.
For every sufficiently small $\tau> 0$, there exists a finite set $F_{\tau} \subset \bR^m$ and
\[
S_{\tau} \in C_c^{\infty}(\bR^m; \bR^m), \qquad \eta_{\tau} \in C_c^{\infty} (\bR^m), \qquad U_{\tau} \in C^{\infty}(\bR^m \setminus F_{\tau} ; \bS^{m-1})
\]
with the following properties.
\begin{enumerate}[(i)]
    \item We have $S_{\tau} = \eta_{\tau} U_{\tau}$ for $\eta_{\tau} \geq 0$, $U_{\tau}$ is a unit vector field, and, as $\tau \downarrow 0$,
    \[
    S_{\tau} \to S, \qquad \textup{div} \, S_{\tau} \to \textup{div} \, S \quad \text{in } \; L^1(\bR^m).
    \]
    \item We have $U_{\tau} = e$ outside a compact set, where $e \in \bS^{m-1}$ is any prescribed vector.
    Also, for every $p \in F_{\tau}$, there exist $r_p > 0$ and $Q_p \in O(m)$ such that
    \[
    U_{\tau}(x) = Q_p \frac{x-p}{|x-p|} \quad \text{for } \; 0 < |x-p| < r_p.
    \]
    The sign $\det Q_p = \pm 1$ is the local degree of the singularity.
    \item For $m=2$, depending on the index $\pm 1$ of the point $p$, we can moreover arrange
    \[
    U_{\tau}(z) = i \frac{z-p}{|z-p|} \quad \text{if $p$ has index $1$}, \qquad U_{\tau}(z) = i \frac{\overline{z-p}}{|z-p|} \quad \text{if $p$ has index $-1$}.
    \]
    Moreover, all but finitely many maximal integral curves of $U_{\tau}$ are defined for all time and are disjoint from $F_{\tau}$.
\end{enumerate}
\end{lemma}
\begin{proof}
    We fix a smooth function $\chi: [0,\infty) \to [0,1]$ with $\chi \equiv 0$ on $[0,3]$ and $\chi \equiv 1$ on $[9,\infty)$, and define $S_{\tau} := \chi \bigl( \frac{|S|}{\tau} \bigr) S$ and $\eta_{\tau} := \chi \bigl( \frac{|S|}{\tau} \bigr) |S|$.
    Because $\chi$ vanishes near zero, these functions are smooth even though $|S|$ itself need not be smooth on $\{ S=0 \}$.
    Moreover, $S_{\tau} \to S$ in $L^1$, and
    \[
    \textup{div} \, S_{\tau} = \chi \bigl( \tfrac{|S|}{\tau} \bigr) \, \textup{div} \, S + \tfrac{1}{\tau} \chi' \bigl( \tfrac{|S|}{\tau} \bigr) \la D |S|, S \rg.
    \]
    The second term is supported on $\{ 3 \tau < |S| < 9 \tau \}$ and is bounded by $\tau^{-1} |\chi'| \, |D S| \, |S| \leq C |D S|$, so
    \[
    |\textup{div} \, S_{\tau} - \textup{div} \, S| \leq C |\nabla S| \, \mathbf{1}_{ \{ |S| < 9 \tau \} }
    \]
    up to a dimensional constant.
    Since $D S=0$ $\cL^m$-a.e.~on $\{ S= 0\}$, we can apply the dominated convergence theorem to obtain $\textup{div} \, S_{\tau} \to \textup{div} \, S$ in $L^1$.

    We now define, at a first step,
    \[
    W^0_{\tau} := \chi \Bigl( \frac{|S|}{\tau/3} \Bigr) \, \frac{S}{|S|} + \Bigl[ 1 - \chi \Bigl( \frac{|S|}{\tau/3} \Bigr) \Bigr] \, e
    \]
    where the first term is interpreted as zero whenever the cutoff vanishes.
    This function is again smooth, for the same reason as $S_{\tau}, \eta_{\tau}$ above.
    Moreover, $W^0_{\tau} = \frac{S}{|S|}$ on $\{ |S| \geq 3 \tau \}$ and $W^0_{\tau} = e$ on $\{ |S| \leq \tau \}$.
    Since $\textup{spt} \, \eta_{\tau} \subset \{ |S| \geq 3 \tau \}$, we can modify $W^0_{\tau}$ in the transition region $ \{ \tau < |S| < 3 \tau \}$ without changing $S_{\tau}$.
    By transversality, we can take an arbitrarily small smooth perturbation $W_{\tau}$ of $W^0_{\tau}$, supported in this transition region, to arrange that $0$ is a regular value of $W_{\tau}$.
    Indeed, let $F^0_{\tau} := (W^0_{\tau})^{-1}(0)$, which is compactly contained in the transition region $\{ \tau < |S| < 3 \tau \}$ since $W^0_{\tau} \neq 0$ outside it.
    We can choose a $\psi \in C_c^{\infty} ( \{ \tau < |S| < 3 \tau \})$ with $\psi \equiv 1$ near $F^0_{\tau}$ and set $W_{\tau} = W_{\tau}^0 + \psi a$ for a sufficiently small $a \in \bR^m$ for which $-a$ is a regular value of $W^0_{\tau}$ on the neighborhood where $\psi = 1$; this is possible by Sard's theorem.
    For $|a|$ sufficiently small, no new zeroes occur where $\psi \neq 1$, because $W^0_{\tau}$ is uniformly bounded away from zero there.
    Hence, $0$ is a regular value of $W_{\tau}$.
    Since $W_{\tau} = e$ outside a compact set, the set $F_{\tau} := W^{-1}_{\tau}(0)$ is finite.

    Finally, we set $U_{\tau} := \frac{W_{\tau}}{|W_{\tau}|}$ on $\bR^m \setminus F_{\tau}$, so $U_{\tau} = \frac{S}{|S|}$ on $\textup{spt} \, \eta_{\tau}$, by construction.
    Therefore, $S_{\tau} = \eta_{\tau} U_{\tau}$ as desired.
    Also, every $p \in F_{\tau}$ is a non-degenerate zero, so the restriction of $U_{\tau}$ to a sufficiently small sphere centered at $p$ has degree $\textup{sgn} \, \det D W_{\tau}(p) \in \{ \pm 1 \}$.
    We choose $Q_p \in O(m)$ to have the same determinant as $DW_{\tau}(p)$.
    Because maps $\bS^{m-1} \to \bS^{m-1}$ of the same degree are homotopic, we can modify $U_{\tau}$ in a small annulus near $p$, and disjoint from $\textup{spt} \, \eta_{\tau}$, so that $U_{\tau}(x) = Q_p \frac{x-p}{|x-p|}$ near $p$.
    This preserves the relation $S_{\tau} = \eta_{\tau} U_{\tau}$ and satisfies the desired properties.

    In dimension $m=2$, we can choose the orthogonal normal forms in the manner specified above,
    \[
    Q_p = J = \begin{pmatrix}
        0 & -1 \\
        1 & 0 
    \end{pmatrix} \quad \text{if $p$ has index $+1$}, \qquad Q_p = J \begin{pmatrix}
        1 & 0 \\
        0 & -1
    \end{pmatrix} \quad \text{if $p$ has index $-1$}.
    \]
    The resulting expressions $U_{\tau}(x) = Q_p \frac{x-p}{|x-p|}$ have precisely the form $i \frac{z-p}{|z-p|}$ or $i \frac{\overline{z-p}}{|z-p|}$ in complex notation, respectively, when $p$ has index $+1$ or $-1$.
    In the first case, for $U_{\tau}(z) = i \frac{z-p}{|z-p|}$, the trajectories near $p$ are circles, so none reaches $p$.
    For $U_{\tau}(z) = i \frac{\overline{z-p}}{|z-p|}$, we can translate to $p=0$ and write $U(x_1, x_2) = \frac{(x_2, x_1)}{\sqrt{x_1^2 + x_2^2}}$ with $\frac{d}{dt} (x_1^2-x_2^2) = 0$.
    Thus, a trajectory can reach $0$ only if $x_1^2 - x_2^2 = 0$, namely along one of the lines $x_2 = \pm x_1$.
    Since there are finitely many points $p \in F_{\tau}$, only finitely many trajectories are excluded.
\end{proof}

\subsection{The smallest eigenvalue functional and projected stresses}\label{section:smallest-eigenvalue}

Let $P$ be an $m$-dimensional Euclidean space.
For $A \in \textup{End}(P)$, we will consider the smallest eigenvalue function
\begin{equation}\label{eqn:smallest-eigenvalue}
    \lambda_m(A) := \lambda_{\min} ( \textup{sym} \, A), \qquad \text{where } \; \textup{sym} \, A := \tfrac{1}{2} (A + A^t).
\end{equation}
Equivalently, $\lambda_m(A) = \min_{ |v| = 1 } \la A v , v \rg$.
We record the following standard properties.
\begin{lemma}\label{lemma:lambda-m-properties}
    The function $\lambda_m$ is concave, positively one-homogeneous, and superadditive on matrices.
    Moreover, $|\lambda_m(A) - \lambda_m(B)| \leq \|A - B\|_{\textup{op}}$ and $\lambda_m (R^t AR) = \lambda_m(A)$ for every $R \in O(P)$.
\end{lemma}

Given a finite $\textup{End}(P)$-valued measure $\mathbf{A}$, let $\sigma$ be any positive measure dominating the entries of $\mathbf{A}$.
We consider the Radon-Nikodym derivative $a := \frac{d \mathbf{A}}{d \sigma}$ and define the functional 
\[
\Lambda_m(\mathbf{A}) := \int \lambda_m(a)\,d\sigma , \qquad a := \frac{d \mathbf{A}}{d \sigma}
\]
which is independent of the choice $\sigma$, due to the $1$-homogeneity of $\lambda_m$.
More generally, let
\begin{equation}\label{eqn:lambda-m-beta}
    \lambda_{m , \beta}(A) := \lambda_m(A) - \beta \, \textup{tr}(A), \qquad \Lambda_{m, \beta} ( \mathbf{A}) := \int \lambda_{m,\beta}(a) \, d \sigma.
\end{equation}
Like $\lambda_m$, the function $\lambda_{m,\beta}$ is concave, positively one-homogeneous, superadditive, invariant under conjugation by orthogonal matrices, and Lipschitz with respect to the operator norm.

Given an $m$-plane $P \in \bG(N,m)$, we fix an isometry $E_P : \bR^m \to P$ and let $\pi_P := E_P^t : \bR^N \to \bR^m$.
We define the projected stress measure $\mathbf{A}_P \in \cM( \bR^m ; \bR^{m \times m})$ by
\begin{equation}\label{eqn:projected-stress-measure}
\begin{split}
    \mathbf{A}_P &:= (\pi_P)_{\#} \bigl[ E_P^t B_{\Psi}(T) E_P \, d V(x,T) \bigr], \qquad \textup{so that} \\
    \int_{\bR^m} \la \varphi(y), d \mathbf{A}_P(y) \rg &\;= \int_{\bR^N \times \bG(N,m)} \la \varphi(\pi_P x), E^t_P B_{\Psi}(T) E_P \rg \, d V(x,T)
\end{split}
\end{equation}
for $\varphi \in C_c(\bR^m ; \bR^{m \times m})$.
Replacing $E_P$ by $E_P R$ for some $R \in O(m)$ conjugates the matrix measure by an orthogonal map, so the corresponding measure is obtained from $\mathbf{A}_P$ by simultaneously rotating the domain and conjugating its values.
In particular, all quantities below involving total variation, divergence, or $\Lambda_{m,\beta}$ are independent of the choice of $E_P$.

For a rectifiable varifold $V$ satisfying $\Theta^m ( \|V\|, x) \geq \theta_0$ for $\|V\|$-a.e.~$x$, we can write $V = v(M, \theta)$ where $\theta \geq \theta_0$ $\cH^m$-a.e.~on the rectifiable set $M$.
Then, $\mathbf{A}_P$ is concentrated on $\pi_P(M)$, and
\[
\cL^m ( \pi_P(M)) \leq \int_M J_m (\pi_P|_{T_x M}) \, d \cH^m(x) \leq \cH^m(M)
\]
by the area formula.
Since $\pi_P$ is continuous, while $M$ is rectifiable and may be chosen Borel, the set $\pi_P(M)$ is Lebesgue measurable.
Thus, we can choose a Borel set $G_P \supset \pi_P(M)$ such that $\mathbf{A}_P$ is concentrated on $G_P$ and $\cL^m (G_P) = \cL^m(\pi_P(M)) \leq \cH^m(M)$.

For a matrix-valued measure $\mathbf{A}$ on $\bR^m$, we define its row-wise distributional divergence by
\[
\la \textup{Div} \, \mathbf{A} , X \rg := - \int_{\bR^m} \la \mathbf{A}, DX \rg, \qquad \text{for } \; X \in C_c^1(\bR^m; \bR^m).
\]
\begin{lemma}\label{lemma:A-P-projection-properties}
Suppose that $\|V\|(\bR^N)<\infty$ and that $V$ has finite anisotropic first variation with respect to an anisotropy $\Psi$ satisfying $\sup_{T \in \bG(N,m)} \|B_{\Psi}(T) - T \| \leq C_N$.
Then, for every $P \in \bG(N,m)$, the measure $\textup{Div} \, \mathbf{A}_P$ has finite total variation and
    \[
    |\textup{Div} \, \mathbf{A}_P| (\bR^m) \leq |\delta_{\Psi} V|(\bR^N), \qquad |\mathbf{A}_P|(\bR^m) \leq C_N \, \|V\|(\bR^N).
    \]
\end{lemma}
\begin{proof}
Let us write $C'_N := \sup_{T \in \bG(N,m)} \|B_{\Psi}(T) \|_{\textup{op}} < \infty$.
For every $X \in C_c^1(\bR^m ; \bR^m)$, the lift $\tilde{X}(x) := E_P X( \pi_P x)$ is bounded and constant along the directions in $P^{\perp}$.
We choose cutoff functions $\eta_R \in C_c^1(\bR^N)$ with $0 \leq \eta_R \leq 1$ and $\eta_R \equiv 1$ on $B_R$, $\eta_R \equiv 0$ outside $B_{2R}$, and $|\nabla \eta_R| \leq CR^{-1}$.
Let
\[
X_R := \eta_R \tilde{X} \in C_c^1(\bR^N; \bR^N), \qquad D \tilde{X}(x) = E_P \, DX( \pi_P x) \, E^t_P.
\]
Thus, the first variation formula for the vector field $X_R$ gives
\begin{equation}\label{eqn:delta-Psi-V-XR}
    \begin{split}
        [\delta_{\Psi} V](X_R) &= \int \eta_R(x) \la B_{\Psi}(T), E_P \, DX (\pi_P x) \, E^t_P \rg \, dV(x,T) \\
        & \quad + \int \la B_{\Psi}(T), \tilde{X}(x) \otimes D \eta_R(x) \rg \, dV(x,T).
    \end{split}
\end{equation}
The second term in the above expression satisfies $|\int \la B_{\Psi}(T), \tilde{X} \otimes D \eta_R \rg \, dV| \leq CC'_N R^{-1} \|X\|_{L^{\infty}} \|V\|(\bR^N)$, and hence tends to zero as $R \to \infty$.
For the first term above, the definition~\eqref{eqn:projected-stress-measure} of $\mathbf{A}_P$ and an application of the dominated convergence theorem give
\[
\lim_{R \to \infty} \int \eta_R(x) \la B_{\Psi}(T), E_P \, DX(\pi_P x) \, E^t_P \rg \, dV(x,T) = \int_{\bR^m} \la DX(y), d \mathbf{A}_P(y) \rg.
\]
On the other hand, 
\[
|\delta_{\Psi} V (X_R)| \leq \| X_R \|_{L^{\infty}} |\delta_{\Psi} V|(\bR^N) \leq \|X \|_{L^{\infty}} |\delta_{\Psi} V| (\bR^N).
\]
Passing to the limit as $R \to \infty$ in~\eqref{eqn:delta-Psi-V-XR} therefore yields
\[
\Bigl| \int_{\bR^m} \la D X, d \mathbf{A}_P \rg \Bigr| \leq \|X \|_{L^{\infty}} |\delta_{\Psi} V|(\bR^N).
\]
This means that the distribution $X \mapsto \la \mathbf{A}_P , X \rg := - \int_{\bR^m} \la D X, d \mathbf{A}_P \rg$ is of order zero, and hence represented by a finite $\bR^m$-valued Radon measure.
Taking the supremum over $X \in C_c^1(\bR^m; \bR^m)$ with $\|X\|_{L^{\infty}}$ proves the claimed bound $|\textup{Div} \, \mathbf{A}_P|(\bR^m) \leq |\delta_{\Psi} V|(\bR^N)$.

Finally, since compression by $E_P$ does not increase the operator norm, we have $\|E^t_P B_{\Psi}(T) E_P \| \leq C'_N$.
Consequently, by the definition of the pushforward measure,
\[
|\mathbf{A}_P| (\bR^m) \leq C_m \int_{\bR^N \times \bG(N,m)} \| E^t_P B_{\Psi}(T) E_P \| \, d V(x,T) \leq C_m C'_N \|V\|(\bR^N)
\]
for a dimensional constant $C_m$.
This completes the proof.
\end{proof}

We next record a consequence of the concavity of the function $\lambda_{m,\beta}$.
In particular, contributions from different
sheets of $V$ lying over the same point of the projection giving rise to $\mathbf{A}_P$ can only improve
the corresponding spectral functional.
\begin{lemma}\label{lemma:concavity-AP}
For every $m$-plane $P \in \bG(N,m)$, it holds that
    \[
    \Lambda_{m, \beta} ( \mathbf{A}_P) \geq \int_{\bR^N \times \bG(N,m)} \lambda_{m,\beta}( E^t_P B_{\Psi}(T) E_P) \, d V(x,T).
    \]
\end{lemma}
\begin{proof}
Consider the projection $\tilde{\pi}_P : \bR^N \times \bG(N,m) \to \bR^m$ given by $\tilde{\pi}_P(x,T) := \pi_P x$, and let $\mu_P := (\pi_P)_{\#} \|V\| = (\tilde{\pi}_P)_{\#} V$.
By the disintegration theorem for rectifiable varifolds~\cite{simon-gmt}*{Ch.~8}, we can find a $\mu_P$-measurable family of probability measures $\{ \sigma_{P,y}\}_{y \in \bR^m}$, with $\sigma_{P,y}$ concentrated on $\tilde{\pi}^{-1}_P(y)$ for $\mu_P$-a.e.~$y$, such that
\begin{equation}\label{eqn:disintegration-of-measures}
    \int f(x,T) \, dV(x,T) = \int_{\bR^m} \Bigl( \int f(x,T) \, d \sigma_{P,y}(x,T) \Bigr) \, d\mu_P(y)
\end{equation}
for every $V$-integrable Borel function $f$.
It follows from the definition~\eqref{eqn:projected-stress-measure} of $\mathbf{A}_P$ and Lemma~\ref{lemma:A-P-projection-properties} that $\mathbf{A}_P \ll \mu_P$, and $\mu_P$-a.e.~$y$ satisfies
\begin{equation}\label{eqn:density-AP-muP}
    \frac{d \mathbf{A}_P}{d \mu_P}(y) = \int_{\tilde{\pi}^{-1}_P(y)} E^t_P \, B_{\Psi}(T) \, E_P \, d \sigma_{P,y}(x,T).
\end{equation}
Indeed, this follows from testing the right-hand side against a compactly supported matrix-valued function and using~\eqref{eqn:disintegration-of-measures}.
Since the definition of $\Lambda_{m,\beta}$ is independent of the choice of dominating measure, and $\mathbf{A}_P \ll \mu_P$, we can use $\mu_P$ in the definition~\eqref{eqn:lambda-m-beta}.
Because $\lambda_{m,\beta}$ is concave and $\sigma_{P,y}$ is a probability measure, Jensen's inequality gives
\begin{align*}
    \Lambda_{m,\beta} ( \mathbf{A}_P) &= \int_{\bR^m} \lambda_{m, \beta} \Bigl( \int E^t_P B_{\Psi}(T) E_P \, d \sigma_{P,y}(x,T) \Bigr) \, d \mu_P(y) \\
    &\geq \int_{\bR^m} \int \lambda_{m,\beta}(E^t_P B_{\Psi}(T) E_P) \, d \sigma_{P,y}(x,T) \, d\mu_P(y) \\
    &= \int_{\bR^N \times \bG(N,m)} \lambda_{m, \beta} (E^t_P B_{\Psi}(T) E_P) \, d V(x,T).
\end{align*}
This completes the proof of our assertion.
\end{proof}

We will study the average of the functional $\Lambda_{m,\beta} ( \mathbf{A}_P)$ over projections onto $m$-planes in $\bG(N,m)$.
We denote by $dP$ the normalized Haar measure on $\bG(N,m)$ and define, for fixed $T \in \bG(N,m)$,
\begin{equation}\label{eqn:d-N-m-beta}
    a_{N,m} := \int_{\mathbb{G}(N,m)} \lambda_{m}(E^t_P T E_P) \, dP, \qquad T \in \mathbb{G}(N,m).
\end{equation}
This quantity is independent of the choice of $T$ because the action of $O(N)$ on $\mathbb{G}(N,m)$ is transitive, while the functional $\lambda_{m}$ is invariant under orthogonal conjugation.
The expression is independent of the choice of isometry $E_P: \bR^m \to P$, so it defines a measurable function of $P$.
\begin{lemma}\label{lemma:averaged-coercivity}
    Suppose that $\sup_{T \in \bG(N,m)} \|B_{\Psi}(T) - T \| \leq \ve$.
    Then, it holds that
    \[
    \int_{\bG(N,m)} \Lambda_{m, \beta} ( \mathbf{A}_P) \, dP \geq ( a_{N,m} - \beta N^{-1} m^2 - (1+m \beta) \ve ) \, \|V\|(\bR^N).
    \]
\end{lemma}
\begin{proof}
Note that $\textup{tr}(E^t_P T E_P) = \textup{tr}(TP)$, and the rotational invariance of the Haar measure gives $\int_{\bG(N,m)} P \, dP = \frac{m}{N} I_N$.
Thus, we compute that
    \begin{align*}
    & \int_{\bG(N,m)} \textup{tr} (E^t_P T E_P) \, dP = \textup{tr} \Bigl( T \int_{\bG(N,m)} P \, dP \Bigr) = \tfrac{m}{N} \, \textup{tr} \, T = \tfrac{m^2}{N}, \\
    & \implies \int_{\bG(N,m)} \lambda_{m, \beta} (E^t_P T E_P) \, dP = a_{N,m} - \beta \tfrac{m^2}{N},
    \end{align*}
    for $a_{N,m}$ the quantity from~\eqref{eqn:d-N-m-beta}.
    The compression $E_P$ does not increase the operator norm, so
    \[
    \| E^t_P ( B_{\Psi}(T) - T) E_P \| \leq \| B_{\Psi}(T) - T \| \leq \ve.
    \]
    Also, the function $\lambda_m$ is $1$-Lipschitz in the operator norm, and $|\textup{tr} \, A| \leq m \|A\|$, so the function $\lambda_{m,\beta}$ is $(1+m\beta)$-Lipschitz.
    Therefore, rearranging terms and integrating over $P \in \bG(N,m)$ gives
    \begin{equation}\label{eqn:lambda-m-beta-inequality}
    \begin{split}
    \lambda_{m, \beta} (E_P^t T E_P ) - (1+m \beta) \ve &\leq \lambda_{m,\beta} (E^t_P B_{\Psi}(T) E_P) \\
    \implies a_{N,m} - \beta N^{-1} m^2 - (1+m \beta) \ve &\leq \int_{\bG(N,m)} \lambda_{m,\beta}(E^t_P B_{\Psi}(T) E_P) \, dP.
    \end{split}
    \end{equation}
    Next, applying Lemma~\ref{lemma:concavity-AP} to the concave, positively one-homogeneous function $\lambda_{m,\beta}$ shows that the integral of the right-hand side over $P \in \bG(N,m)$ is bounded from above by $\Lambda_{m, \beta}(\mathbf{A}_P)$, for Haar-a.e.~$P$.
    Integrating in $P$, applying Fubini's theorem, and using~\eqref{eqn:lambda-m-beta-inequality}, we therefore obtain
    \begin{align*}
        \int_{\bG(N,m)} \Lambda_{m, \beta} ( \mathbf{A}_P) \, dP &\geq \int_{\bG(N,m)} \Bigl(\int_{\bR^N \times \bG(N,m)} \lambda_{m, \beta} (E^t_P B_{\Psi}(T) E_P) \, dV(x,T) \Bigr) \, dP \\
        &\geq \int_{\bR^N \times \bG(N,m)} \Bigl( \int_{\bG(N,m)} \lambda_{m,\beta}(E^t_P B_{\Psi}(T) E_P) \, dP \Bigr) \, dV(x,T) \\
        &\geq ( a_{N,m} - \beta N^{-1}m^2 - (1+m\beta) \ve ) \, \|V\|(\bR^N)
    \end{align*}
    using the bound~\eqref{eqn:lambda-m-beta-inequality}.
    This completes the proof.
\end{proof}

For anisotropic $2$-varifolds, we introduce a functional $\Psi_2$ on $2 \times 2$-matrix-valued Radon measures.
This construction will enable us to control $2$-varifolds through their one-dimensional projections, while retaining the properties of $\lambda_{2,\beta}$ discussed above.
\begin{lemma}\label{lemma:2x2-matrix}
    Let $A$ be a $2 \times 2$ matrix with entries $(a_{ij})$.
    Then, we have
    \begin{equation}\label{eqn:psi(A)-equality}
\begin{split}
    \psi_2(A) &:= \min \{ ( a_{11} - |a_{12}|)_+ , (a_{22} - |a_{21}|)_+ \} - (a_{11} - |a_{12}|
    )_- - (a_{22} - |a_{21}|)_{_-} \\
    &= \min \{ a_{11} - |a_{12}|,  a_{22} - |a_{21} | , (a_{11} - |a_{12}|) + (a_{22} - |a_{21}|) \}.
\end{split}
\end{equation}
Moreover, $\psi_2$ is a concave and positively one-homogeneous function, with $\psi_2(A+B) \geq \psi_2(A) + \psi_2(B)$ for every pair of matrices $A,B$.
\end{lemma}
\begin{proof}
    Let $s(A) := a_{11} - |a_{12}|$ and $t(A) := a_{22} - |a_{21}|$, so the above equality amounts to proving
\[
\min \{ s_+ , t_+ \} - s_- - t_- = \min \{ s , t , s+t \}
\]
for any real numbers $s,t$.
This equality follows by considering the possible sign of $s,t$; by symmetry, we need only address $\{ s,t \geq 0 \}$, $\{ s,t \leq 0 \}$, and $\{ t \leq 0 \leq s\}$.
In the first case, both expressions equal $\min \{ s,t \}$; in the second case, they equal $s+t$; in the last case, they both equal $t$.
This proves the equality~\eqref{eqn:psi(A)-equality}.
Thus, we can express $\psi_2(A) = \min \{ s(A), t(A), s(A) + t(A) \}$, where each function in this minimization is concave, positively one-homogeneous in $A$, and superadditive on matrices.
Therefore, $\psi_2$ has the same properties; this proves our assertion.
\end{proof}

We now write $\textup{skew} \, A := \frac{1}{2} (A - A^t)$.
As for the function $\Lambda_{m,\beta}$ of~\eqref{eqn:lambda-m-beta}, we define functions of matrix-valued measures $\mathbf{A}$ by
\[
\cN_m(\mathbf{A}) := \int \lambda_m \Bigl( \frac{d \mathbf{A}}{d \sigma} \Bigr)_- \, d \sigma, \qquad \cK_m (\mathbf{A}) := \int \Bigl\| \textup{skew} \, \frac{d \mathbf{A}}{d \sigma} \Bigr\|_{\textup{op}} \, d \sigma
\]
where $\sigma$ is any measure dominating all the entries of $\mathbf{A}$, with $\frac{d \mathbf{A}}{d \sigma}$ the Radon-Nikodym derivative.
We also define, specifically for $2\times2$ matrix-valued measures $\mathbf{A} \ll \sigma$,
\[
\Psi_2( \mathbf{A}) := \int \psi_2 \Bigl( \frac{d \mathbf{A}}{d \sigma} \Bigr) \, d \sigma.
\]
Because the functions $\lambda_m(a)_- , \| \textup{skew} \, a \|_{\textup{op}}$, and $\psi_2(a)$ are positively one-homogeneous, these functionals are well-defined independently of the particular choice of $\sigma$.
\begin{lemma}\label{lemma:symmetric-positive-definite-average}
    Let $S$ be a symmetric positive semidefinite $2 \times 2$ matrix.
    Then,
    \[
    \frac{1}{2 \pi} \int_0^{2 \pi} \psi_2 (R_{\theta}^t S R_{\theta}) \, d \theta = \frac{4}{\pi} \lambda_{2,\beta}(S), \qquad \text{where } \; \beta := \frac{4-\pi}{8}.
    \]
    Here, $R_{\theta}\in \textup{SO}(2)$ denotes the rotation matrix of angle $\theta$.
\end{lemma}
\begin{proof}
    Let $\lambda_1 \geq \lambda_2 \geq 0$ be the eigenvalues of the matrix $S$.
    After rotating the basis, we can write
    \begin{equation}\label{eqn:a11-a22-a12}
    \min \{ a_{11}  , a_{22} \} = \frac{\lambda_1 + \lambda_2}{2} - \frac{\lambda_1 - \lambda_2}{2} \, |\cos 2 \theta|, \qquad |a_{12}| = \frac{\lambda_1 - \lambda_2}{2} |\sin 2 \theta|.
    \end{equation}
    We recall that a symmetric $2 \times 2$ matrix $S$ with entries $(a_{ij})$ is positive semidefinite if and only if $a_{11}, a_{22} \geq 0$ and $a_{11} a_{22} \geq a_{12}^2$.
    Hence, at least one of $a_{11} - |a_{12}|$ and $a_{22} - |a_{21}|$ is non-negative, so the definition of $\psi_2$ from Lemma~\ref{lemma:2x2-matrix} gives $\psi_2(S) = \min \{ a_{11}, a_{22} \} - |a_{12}|$.
    Since the averages of the functions $|\cos 2 \theta|$ and $|\sin 2 \theta|$ over a period $[0,2\pi]$ both equal $\frac{2}{\pi}$, the relation~\eqref{eqn:a11-a22-a12} implies that
    \[
    \frac{1}{2 \pi} \int \psi_2 ( R_{\theta}^t S R_{\theta}) \, d \theta = \frac{\lambda_1 + \lambda_2}{2} - \frac{2}{\pi} (\lambda_1 - \lambda_2) = \Bigl( \frac{1}{2} - \frac{2}{\pi} \Bigr) (\lambda_1 + \lambda_2) + \frac{4}{\pi} \lambda_2.
    \]
    Since $\lambda_1 + \lambda_2 = \textup{tr}(S)$ and $\lambda_{\min}(S) = \lambda_2$ this proves our claim.
\end{proof}

\begin{lemma}\label{lemma:signed-average}
For every $2 \times 2$ matrix $A$, it holds that
\[
\frac{4}{\pi} \lambda_{2,\beta}(A) \leq \frac{1}{2 \pi} \int_0^{2 \pi} \psi_2 (R^t_{\theta} A R_{\theta}) \, d \theta + \lambda_2(A)_- + 2 \, \| \textup{skew} \, A \|_{\textup{op}} .
\]
\end{lemma}
\begin{proof}
We set $\ell := \lambda_2(A)_-$.
Let us decompose
\[
A =  S+kJ , \qquad \textup{where } \; S := \textup{sym} \, A, \qquad J= \begin{pmatrix}
    0 & -1 \\
    1 & 0 
\end{pmatrix}.
\]
Then, $\| \textup{skew} \, A \|_{\textup{op}} = |k|$.
Next, by definition, the matrix $S + \ell I \geq 0$ is positive semidefinite and satisfies $\lambda_2 ( S+ \ell I) = \lambda_2(S) + \ell$, so we can write
\[
\lambda_{2,\beta}(S+ \ell I) = \lambda_{2,\beta}(S) + (1-2 \beta) \ell = \lambda_{2,\beta}(S) + \tfrac{\pi}{4} \ell.
\]
Also, the projections $s(B) := a_{11} - |a_{12}|$ and $t(B) := a_{22} - |a_{21}|$ increase by $\ell$, namely $s(B+ \ell I) = s(B) + \ell$ and $t(B+ \ell I) = t(B)+\ell$.
The definition of $\psi_2$ implies that $\psi_2(B+ \ell I) \leq \psi_2(B) + 2 \ell$.
We apply these inequalities with $B = R^t_{\theta} (S+ \ell I ) R_{\theta}$ and average over $\theta \in [0,2 \pi]$ to obtain
\[
\frac{4}{\pi} \lambda_{2,\beta}(S+ \ell I) = \frac{1}{2 \pi} \int_0^{2\pi} \psi_2(R^t_{\theta} (S+ \ell I) R_{\theta} ) \, d \theta \leq \frac{1}{2\pi} \int_0^{2\pi} \psi_2(R^t_{\theta} S R_{\theta}) \, d \theta + 2\ell
\]
due to $R^t_{\theta} (S +\ell I) R_{\theta} = R^t_{\theta} S R_{\theta} + \ell I$.
Since $\lambda_{2,\beta}(S+\ell I) = \lambda_{2,\beta}(S) + \frac{\pi}{4} \ell$, this inequality becomes
\[
\frac{4}{\pi} \lambda_{2,\beta}(S) + \ell \leq \frac{1}{2 \pi} \int_0^{2 \pi} \psi_2( R^t_{\theta} S R_{\theta}) \, d \theta + 2 \ell
\]
and rearranging terms proves the inequality with right-hand side $\lambda_2(A)_-$.
Next, conjugation by elements of $\textup{SO}(2)$ leaves the matrix $kJ$ unchanged.
Adding $kJ$ to a matrix changes the absolute value of each off-diagonal entry by at most $|k|$, and therefore decreases $\psi_2$ by at most $2|k|$.
Hence,
\[
\psi_2 ( R^t_{\theta} S R_{\theta}) \leq \psi_2(R^t_{\theta} A R_{\theta}) + 2 \, |k| = \psi_2(R^t_{\theta} A R_{\theta}) + 2 \, \| \textup{skew} \, A \|_{\textup{op}}.
\]
Since $\lambda_{2,\beta}(A) = \lambda_{2,\beta}(S)$ due to $\textup{sym} \, A = S$ and $\textup{tr} \, A = \textup{tr} \, S$, the result follows.
\end{proof}

\section{A Kakeya-type inequality for vector fields}\label{section:kakeya}

In this section, we will obtain a Kakeya-type inequality for vector fields on the plane.
The result strengthens certain steps of~\cite{anisotropic-michael-simon}*{Theorem 3.1} with sharp constants, which are crucial for our subsequent arguments.
We first extend the construction of cones and flow boxes as in~\cite{anisotropic-michael-simon}*{\S 3.2} to arbitrary dimension.
Given a unit direction $e \in \bS^{m-1}$, we can decompose every vector as $v = \la v,e \rg e + v_e^{\perp}$.
We then define an open $45^{\circ}$ cone
\begin{equation}\label{eqn:q-cone}
    q_e(v) := \la v, e \rg - |v_e^{\perp}|, \qquad C^e := \{ v : q_e(v) > 0 \}
\end{equation}
about the oriented direction $e$.
Given a vector-valued measure $\mu \in \cM ( \bR^m ; \bR^m)$, we choose any positive Radon measure $\sigma$ dominating $\mu$ entry-wise, denoted $\mu \ll \sigma$.
We then define
\begin{equation}\label{eqn:P-N-measures}
    \mu^+_e := \Bigl[ q_e \Bigl( \frac{d \mu}{d \sigma} \Bigr) \Bigr]_+ \sigma, \qquad \mu^-_e := \Bigl[ q_e \Bigl( \frac{d \mu}{d \sigma} \Bigr) \Bigr]_- \sigma,
\end{equation}
which are independent of $\sigma$ because the functions $q_e$ and $(q_e)_{\pm}$ are positively one-homogeneous.

\begin{theorem}\label{thm:kakeya-type-inequality}
    Let $S,T$ be finite vector-valued measures on $\bR^2$ with measure-valued divergence.
    For every non-negative compactly supported Borel function $\chi$, it holds that
    \begin{align*}
        \int_{\bR^2} \chi \min \{ q_x(S)_+, q_y(T)_+ \} &\leq \| \chi \|_{L^2} \Bigl( \int_{\bR^2} |S| + |\textup{div} \, S| \Bigr)^{\frac{1}{2}} \Bigl( \int_{\bR^2} |T| + |\textup{div} \, T| \Bigr)^{\frac{1}{2}} \\
        & \quad + 2 \| \chi \|_{L^{\infty}} \int_{\bR^2} ( |\textup{div} \, S|  +|\textup{div} \, T|) + \| \chi \|_{L^{\infty}} \int_{\bR^2} (q_x(S)_- + q_y(T)_-).
    \end{align*}
\end{theorem}
The proof of Theorem~\ref{thm:kakeya-type-inequality} will proceed in a number of steps, following the general strategy of De Philippis and Pigati.
The key improvement over their result is the sharp factor $\| \chi \|_{L^{\infty}}$ for the term $\int_{\bR^2} ( q_x(S)_- + q_y(T)_-)$, in place of the dimensional constant $C'$ used in their work.
In the proof of Proposition~\ref{prop:A-div-S} towards Theorem~\ref{thm:anisotropic-MS}, we will invoke this result with $\chi = \mathbf{1}_{E \cap B_R}$ having $\| \chi \|_{L^{\infty}} \leq 1$.
Crucially, this sharp constant will allow us to collect terms as
\[
\Psi_2( \mathbf{A}) = \int_{\bR^2} \min \{ q_x(S)_+ , q_y(T)_+ \} - \int_{\bR^2} (q_x(S)_- + q_y(T)_-).
\]
and apply the averaged spectral bound of Lemma~\ref{lemma:2x2-matrix} for $\Lambda_{2,\beta}( \mathbf{A})$.

Let us outline the proof of Theorem~\ref{thm:kakeya-type-inequality}.
As in Lemma~\ref{lemma:lipschitz-curve}, in Euclidean space $\bR^m$ of arbitrary dimension, we can consider the strips
\[
\cU^e_j := \{ x \in \bR^m : j \leq \la x, e \rg \leq j+1 \}, \qquad j \in \bZ.
\]
Given a smooth nowhere-vanishing vector field $S$, a \textbf{flow box in the $e$-direction} is a region $\cR^e \subset \cU^e_j$, diffeomorphic to $\Sigma \times (0,1)$ and foliated by integral curves of $S$, each entering transversely through the side $\{ \la x, e \rg = j \}$, exiting transversely through the side $\{ \la x,e \rg = j+1 \}$, and with the lateral boundary foliated by integral curves of $S$.
These definitions recover the notation $\cU^x_j,C^x$ and $\cU^y_k, C^y$ for $e_x = (1,0)$ and $e_y = (0,1)$, in which case $q_x(v) = v_x - |v_y|$ and $q_y(w) = w_y - |w_x|$ on $\bR^2$.
See also~\cite{anisotropic-michael-simon}*{Definition 3.7}, where the terminology \textit{good region} is used for the vertical and horizontal flow boxes $\cR^x, \cR^y$.
We refer the reader to~\cite{anisotropic-min-max}*{\S 3.1} for illustrations of the flow boxes formed by vector fields $S,T$ on the plane.

To prove Theorem~\ref{thm:kakeya-type-inequality}, we will use repeated regularization.
We first establish the estimate for smooth compactly supported vector fields.
The passage from general vector-valued measures to this case is performed at the end of the proof: we convolve the $\bR^4$-valued measure $(S,T)$ with a common approximate identity, use the uniform estimate for the resulting smooth vector fields, and invoke Reshetnyak's continuity theorem~\ref{lemma:reshetnyak-continuity}.
We also use a more precise description after Lemma~\ref{lemma:S-vector-field}.
Given vector fields $S,T$, we apply Lemma~\ref{lemma:S-vector-field} separately to $S$ and $T$, prescribing the directions $e=e_x$ and $e=e_y$ in part $(ii)$, respectively.
Thus, for sufficiently small $\tau>0$, we obtain $S_{\tau} = \eta_{\tau} U_{\tau}$ and $T_{\tau} = \zeta_{\tau} V_{\tau}$, where $\eta_{\tau}, \zeta_{\tau} \geq 0$, the unit fields $U_{\tau}, V_{\tau}$ are $C^{\infty}$ away from finite sets $F^S_{\tau}, F^T_{\tau}$, and
\begin{equation}\label{eqn:S-T-convergence}
S_{\tau} \to S, \qquad T_{\tau} \to T, \qquad \textup{div} \, S_{\tau} \to \textup{div} \, S, \qquad \textup{div} \, T_{\tau} \to \textup{div} \, T \quad \text{in } \; L^1(\bR^2)
\end{equation}
as $\tau \downarrow 0$.
The functions $v \mapsto q_x(v)_{\pm}$ and $v \mapsto q_y(v)_{\pm}$ are Lipschitz and positively one-homogeneous, so all terms in the desired estimate pass to the limit under the $L^1$-convergence~\eqref{eqn:S-T-convergence}.
Hence, it suffices to prove the estimate for the vector fields $S_{\tau}, T_{\tau}$, with constants independent of $\tau$.
In what follows, we will fix and suppress the subscript $\tau$ and write $S = \eta U$ and $T = \zeta V$, with $F = F_{\tau} := F_{\tau}^S \cup F^T_{\tau}$ the union of the corresponding finite exceptional sets.

Thus, it remains to prove a uniform estimate for fields of the form $S = \eta U$ and $T = \zeta V$, where $U,V$ are the global unit direction fields furnished by Lemma~\ref{lemma:S-vector-field}
We separately estimate the contributions of $q_x(S)$ and $q_y (T)$ inside the flow boxes $\cR^x$ (respectively, $\cR^y$) and away from them.
The region $\cR^x$ is foliated by integral curves of $S$, so expressions of the form $\int_{\cR^x} (q_x (S))_+$ can be estimated using the one-dimensional integral estimates of Lemma~\ref{lemma:rising-sum} along trajectories.
In the region $\bR^2 \setminus \cR^x$, we use Lemma~\ref{lemma:lipschitz-curve} in Smirnov's decomposition Theorem~\ref{thm:smirnov}.

\subsection{Bounds inside flow boxes}
We begin with a description of the flow boxes $\cR^e \subset \cU^e_j$.
\begin{lemma}\label{lemma:flow-box}
    Let $\Omega \subset \bR^m$ be a flow box for a smooth non-vanishing vector field $S$, namely a region foliated by integral curves of $S$ and admitting flow coordinates $F: \Sigma \times I \to \Omega$.
    Then, there exist a smooth positive function $\alpha$ and a divergence-free, nowhere-vanishing vector field $Z$ such that $S = \alpha Z$ and the integral curves of $Z$ and $S$ coincide, with the same orientation.
\end{lemma}
\begin{proof}
    We denote by $F(p,t)$ the flow coordinates for $S$, where $p \in \Sigma$ belongs to a transverse section.
    We define a smooth function $h>0$ and a vector field $Z := hS$ by solving, along characteristics,
    \begin{align*}
    &\tfrac{d}{dt} h(F(p,t)) = - h(F(p,t)) \, \textup{div} \, S(F(p,t)), \qquad h(F(p,0)) = 1, \\
    & \implies \textup{div} \, Z = \la S, \nabla h \rg + h \, \textup{div} \, S = 0
    \end{align*}
    by the construction of $h$.
    Finally, taking $\alpha := h^{-1}$ proves the claim.
\end{proof}

By Lemma~\ref{lemma:flow-box}, on a flow box $\cR^e \subset \cU^e_j$ for $S$, we write $S = \alpha Z$ with $\textup{div} \, Z=0$.
For each trajectory $\Gamma$ crossing the box, set $\alpha_f|_{\Gamma} := \min_{\Gamma} \alpha$ and decompose $S = S_f + S_d$ with $S_f := \alpha_f Z$ a continuous vector field.
Since $\alpha_f$ is constant along trajectories, this implies that $\textup{div} \, S_f =0$ and $\textup{div} \, S_d = \textup{div} S$.
For $p \in \cR^e$, let $\gamma_p([0,\tau_p])$ be the forward integral curve of $S$ through $p$, stopped when it reaches the right side of the flow box.
We define a \textbf{good set} as the Borel set of points
\[
G^e := \{ p \in \cR^e : \gamma_p(t) - p \in C^e \; \text{for every later point of the curve}, \; 0 < t \leq \tau_p \}.
\]
\begin{proposition}\label{prop:generalized-flow-boxes}
    In the above setting and using the notation~\eqref{eqn:q-cone}, we have
    \begin{align*}
    \int_{\cR^e \setminus G^e} (q_e (S_f))_+ &\leq \int_{\cR^e} (q_e(S_f))_-, \\
    \int_{\cR^e} (q_e(S_d))_+ &\leq \int_{\cR^e} |\textup{div} \, S| + \int_{\cR^e} (q_e(S_d))_-.
    \end{align*}
\end{proposition}
\begin{proof}
    We parametrize each integral curve $\gamma : [0,\tau] \to \cR^e$ of $S$ (hence, also of $Z$) by arc-length in the orientation of $Z$ and set $g(t) := q_e (\dot{\gamma}(t))$.
    We will apply the rising sun Lemma~\ref{lemma:rising-sum} to the function $\xi(t) := \int_0^t g(s) \, ds$.
    In the notation of the Lemma, $s \in E$ implies that $\gamma(s) \in G^e$ above, as
    \[
    0 < \xi(t) - \xi(s) \leq \la \gamma(t), e \rg - \la \gamma(s), e \rg - | \gamma(t)_e^{\perp} - \gamma(s)^{\perp}_e |,
    \]
    so $\gamma(t) - \gamma(s) \in C^e$ and $\gamma(s) \in G^e$.
    Therefore, the portion of the level curve outside $G^e$ is contained in $[0,\tau] \setminus E$.
    Therefore, Lemma~\ref{lemma:rising-sum} gives
    \[
    \int_{\gamma \setminus G^e} g_+ \leq \int_{[0,\tau] \setminus E} (\xi')_+ \leq \int_0^{\tau} (\xi')_- = \int_{\gamma} g_-.
    \]
    Because the function $\alpha_f$ is positive and constant along each trajectory, this also implies that 
    \begin{equation}\label{eqn:alpha-f-g-plus}
        \int_{\gamma \setminus G^e} \alpha_f g_+ \leq \int_{\gamma} \alpha_f g_-
    \end{equation}
    The vector field $Z$ is divergence-free and the region $\cR^e$ is foliated by integral curves of $Z$, so a change of variables along characteristics produces the disintegration
    \begin{equation}\label{eqn:cancel-factor-Z}
    \int_{\cR^e} H(x) \, dx = \int_{\Sigma} \int_{\gamma_q} \frac{H}{|Z|} \, d \cH^1 \, d \nu(q), \qquad H \in L^1(\cR^e),
    \end{equation}
    for every non-negative Borel function $H \in L^1(\cR^e)$.
    Here, $\Sigma$ denotes an entrance cross-section for integral curves along $\cR^e$, $d \nu = |\la Z, n \rg| \, d \cH^{m-1}$, and the transverse measure flux is preserved through cross-sections due to $\textup{div} \, Z = 0$.
    Recalling that $(q_e(S_f))_{\pm} = \alpha_f |Z| \, (q_e(\dot{\gamma}))_{\pm} = \alpha_f |Z| \, g_{\pm}$ by the definition of $S_f$, we may  therefore cancel the factor $|Z|$ in~\eqref{eqn:cancel-factor-Z} and integrate the inequality~\eqref{eqn:alpha-f-g-plus} along trajectories.
    This proves the first claimed inequality.
    
    Regarding the second inequality, we apply the weighted integral bound in the second part of Lemma~\ref{lemma:lipschitz-curve} to the function $g(t) := q_e( \dot{\gamma}(t))$ with $h := \alpha_d = (\alpha - \alpha_f) \in W^{1,1}([0,\tau])$, which is non-negative and satisfies $\min_{\gamma} \alpha_d = 0$ by the definition of $\alpha_f = \min_{\gamma} \alpha$.
    This implies that
    \begin{equation}\label{eqn:int-tau-alpha-d}
        \int_0^{\tau} \alpha_d g_+ \leq \int_0^{\tau} |\dot{\alpha}_d| + \int_0^{\tau} \alpha_d g_- \, .
    \end{equation}
    Since $\dot{\gamma} = \frac{Z}{|Z|}$ and $\textup{div} \, Z=0$, while $\alpha_f$ is constant along trajectories, we compute that
    \[
    \dot{\alpha}_d = \Bigl\langle \frac{Z}{|Z|}, \nabla \alpha \Bigr\rangle, \qquad
    |\dot{\alpha}| = \frac{|\la Z, \nabla \alpha_d \rg|}{|Z|} = \frac{|\textup{div} \, S|}{|Z|}
    \]
    due to $\textup{div} \, S = \la Z, \nabla \alpha \rg$.
    Thus, we can multiply the inequality~\eqref{eqn:int-tau-alpha-d} by the coarea measure and use the disintegration formula~\eqref{eqn:cancel-factor-Z} along the flow of $Z$ to see that the three terms become, respectively, 
    \[
    \int_{\cR^e} (q_e(S_d))_+ \leq \int_{\cR^e} |\textup{div} \, S| + \int_{\cR^e} (q_e(S_d))_-.
    \]
    This completes the proof of our result.
\end{proof}

We now specialize the above bounds to the case of $\bR^2$.
When $e_x = (1,0)$, we call $\cR^x := \cR^{(1,0)}$ a vertical flow box; when $e_y = (0,1)$, we call $\cR^y := \cR^{(0,1)}$ a horizontal flow box.
We also write 
\[
q_x(S_f) := q_{(1,0)}(S_f), \quad q_y(T_f) := q_{(0,1)} (T_f), \quad G^x := G^{(1,0)}, \quad G^y := G^{(0,1)}.
\]
\begin{lemma}\label{lemma:good-good-interaction}
    Let $\cR^x, \cR^y \subset \bR^2$ be respectively a vertical flow box for $S$ and a horizontal flow box for $T$, with $S_f, T_f$ and $G^x, G^y$ as above, with $G := G^x \cap G^y$.
    Then, for every $\chi \in L^2(\bR^2)$, we have
    \[
    \int_G |\chi| \, \sqrt{ q_x(S_f)_+ \, q_y (T_f)_+ } \leq \| \chi \|_{L^2} \left[ \int_{\cR^x} ( |S| + |\textup{div} \, S|) \right]^{\frac{1}{2}} \left[ \int_{\cR^y} ( |T| + |\textup{div} \, T|) \right]^{\frac{1}{2}}.
    \]
\end{lemma}
\begin{proof}
The proof is similar to that of Theorem 3.1 in~\cite{anisotropic-michael-simon}*{\S 3.4}, but we estimate terms more carefully to obtain the sharp bound by $\| \chi \|_{L^2}$ instead of a large cumulative constant $C$.
Since the vector fields $S_f, T_f$ are divergence-free and the flow boxes $\cR^x, \cR^y$ are simply connected, we can find functions $\varphi \in C^1(\cR^x) , \psi \in C^1(\cR^y)$ such that $S_f = \nabla^{\perp} \varphi$ and $T_f = \nabla^{\perp} \psi$, where $\nabla^{\perp} = ( - \partial_y, \partial_x)$.
Each trajectory of $S_f$ is contained in a level set of $\varphi$, and $\la S_f , e_x \rg = - \partial_y \varphi > 0$ on the left side of $\cR^x$, so $\varphi$ is strictly monotone there.
Since $\la S_f, e_x \rg \geq 0$ on the entrance side, $\varphi$ is monotone there.

We recall the expressions $S = \alpha Z$ and $T=\beta W$.
Because any two distinct trajectories where $\alpha_f > 0$ have distinct $\varphi$ values for the set $G_+ := G\cap\{\alpha_f\beta_f>0\}$, equality of $\varphi$ (resp.~$\psi$) is equivalent to points lying on the same $S$-trajectory (resp.~$T$-trajectory).
The integrand vanishes on $G\setminus G_+$, so it suffices to apply the area formula on $G_+$.
Since every trajectory meets the left side exactly once, each level set of $\varphi$ in $\cR^x$ is a single $S$-trajectory.
Likewise, level sets of $\psi$ in $\cR^y$ are $T$-trajectories.

We also observe that the map $\Phi = (\varphi, \psi)$ is injective on $G = G^x \cap G^y$.
Indeed, if $\Phi(p) = \Phi(\tilde{p})$, then $p,\tilde{p}$ lie on the same $S$-trajectory and on the same $T$-trajectory.
If $p \neq \tilde{p}$, the definition of $G^x$ gives $\tilde{p}-p \in C^x$ or $p-\tilde{p} \in C^x$, while the definition of $G^y$ gives $\tilde{p}-p \in C^y$ or $p-\tilde{p} \in C^y$.
Thus,
\[
|p_y - \tilde{p}_y| < |p_x - \tilde{p}_x|, \qquad |p_x - \tilde{p}_x| < |p_y - \tilde{p}_y|,  
\]
must hold simultaneously; this is impossible.
Thus, the map $\Phi = (\varphi, \psi)$ is injective.

For $v \in \overline{C^x}$ and $w \in \overline{C^y}$, we write $v = (v_x,v_y)$ and $w = (w_x, w_y)$ and observe that
\begin{equation}\label{eqn:det(v,w)}
    \det (v,w) = v_x w_y - v_y w_x \geq v_x w_y - |v_y| \, |w_x| \geq (v_x - |v_y|) \, (w_y - |w_x|) = q_x(v) \, q_y(w).
\end{equation}
At every point $p \in G = G^x \cap G^y$, we have $S_f(p) \in \overline{C^x}$ and $T_f(p) \in \overline{C^y}$.
Indeed, if $p \in G^x$, then $\gamma_p(t) - p \in C^x$ for every sufficiently small $t>0$, so
\[
| \la \gamma(t), e_y \rg - \la p, e_y \rg | < \la \gamma(t), e_x \rg - \la p, e_x \rg.
\]
Dividing by $t>0$ and sending $t \downarrow 0$ shows that $| \la \dot{\gamma}(0), e_y \rg | \leq \la \dot{\gamma}(0), e_x \rg$, hence $\dot{\gamma}(0) \in \overline{C^x} = \{ (v_x, v_y) : |v_y| \leq v_x \}$.
Since $S_f$ is a positive multiple of the oriented tangent $\dot{\gamma}$, it satisfies $S_f(p) \in \overline{C^x}$.
The proof for $T_f(p) \in \overline{C^y} = \{ (w_x, w_y) : |w_x| \leq w_y \}$ is analogous.
Therefore,~\eqref{eqn:det(v,w)} implies that
\begin{equation}\label{eqn:qx-Sf}
\begin{split}
&q_x(S_f)_+ \, q_y(T_f)_+ \leq \det (S_f, T_f) \quad \text{on } \; G_+ \\
& \implies \int_G  q_x(S_f)_+ \, q_y(T_f)_+ = \int_{G_+}  q_x(S_f)_+ \, q_y(T_f)_+ \leq \int_{G_+} \det (S_f, T_f) 
\end{split}
\end{equation}
because $q_x(S_f)_+ q_y(T_f)_+ = 0$ on $G \setminus G_+$ due to $\alpha_f \beta_f=0$ there.
Moreover, $\det D \Phi = \det (S_f, T_f) \geq 0$ on $G$.
By the injectivity of $\Phi$ and the area formula, we bound
\begin{equation}\label{eqn:injectivity-and-area}
    \int_{G_+} \det (S_f, T_f) = \cL^2 ( \Phi(G)) \leq \textup{osc}_{\cR^x} \varphi \, \textup{osc}_{\cR^y} \psi.
\end{equation}
Because $\varphi$ is constant along the trajectories of $S$, its range on $\cR^x$ equals its range on the left entrance boundary $\partial_L \cR^x$.
Writing $S_f = \alpha_f Z$ and $S = \alpha Z$, with $0 < \alpha_f \leq \alpha$, while $\la Z, e_x \rg > 0$ on the entrance side, we can estimate $\textup{osc}_{\cR^x} \varphi$ as
\begin{equation}\label{eqn:osc-varphi-psi}
    \begin{split}
        \textup{osc}_{\cR^x} \varphi &= \int_{\partial_L \cR^x} |\partial_y \varphi| \, d \cH^1 = \int_{\partial_L \cR^x} \la S_f, e_x \rg \, d \cH^1 \leq \int_{\partial_L \cR^x} \la S, e_x \rg \, d \cH^1, \\
        \textup{osc}_{\cR^y} \psi &= \int_{\partial_B \cR^y} |\partial_x \psi| \, d \cH^1 = \int_{\partial_B \cR^y} \la T_f, e_y \rg \, d \cH^1\leq \int_{\partial_B \cR^y} \la T, e_y \rg \, d \cH^1.
    \end{split}
\end{equation}
The second bound here is obtained by the analogous argument, where $\partial_B \cR^y$ denotes the bottom entering boundary of the horizontal flow box $\cR^y$.
To estimate the entrance flux in the first case, for a.e.~$s \in [j,j+1]$ we can apply the divergence theorem $\cR^x \cap \{ j < x < s\}$ and use the fact that the lateral boundary is tangent to $S$.
This gives
\begin{equation}\label{eqn:entrance-flux-term}
    \begin{split}
    & \int_{\partial_L \cR^x} \la S, e_x \rg = \int_{ \cR^x \cap \{ x=s\} } \la S, e_x \rg - \int_{ \cR^x \cap \{ j < x < s \} } \textup{div} \, S \\
    & \implies \int_{\partial_L \cR^x} \la S, e_x \rg \leq \int_{\cR^x} ( |\la S, e_x \rg| + |\textup{div} \, S|) \leq \int_{\cR^x} ( |S| + |\textup{div} \, S|)
    \end{split}
\end{equation}
where the implication follows from averaging over $s$ over the unit intervals $[j, j+1]$.
The estimate for the bottom flux term $\int_{\partial_B \cR^y} \la T, e_y \rg \, d \cH^1$ is analogous.
Combining the bounds~\eqref{eqn:qx-Sf} - \eqref{eqn:entrance-flux-term} and using the Cauchy-Schwarz inequality, we therefore obtain
\allowdisplaybreaks{
\begin{align*}
    \int_G |\chi| \, \sqrt{q_x(S_f)_+ \, q_y(T_f)_+} &\leq \| \chi \|_{L^2} \Bigl( \int_G q_x(S_f)_+ \, q_y(T_f)_+ \Bigr)^{\frac{1}{2}} \leq \| \chi \|_{L^2} \left( \int_G \det (S_f, T_f) \right)^{\frac{1}{2}} \\
    &\leq \| \chi\|_{L^2} \left( \int_{\partial_L \cR^x} \la S, e_x \rg \right)^{\frac{1}{2}} \left( \int_{\partial_B \cR^y} \la T, e_y \rg \right)^{\frac{1}{2}} \\
    &\leq \| \chi \|_{L^2} \left( \int_{\cR^x} ( |S| + |\textup{div} \, S|) \right)^{\frac{1}{2}} \left( \int_{\cR^y} ( |T| + |\textup{div} \, T|) \right)^{\frac{1}{2}} .
\end{align*}}
This completes the proof.
\end{proof}

\subsection{Bounds outside flow boxes}

Next, we examine the decomposition $S_{\tau} = \eta_{\tau} U_{\tau}$ and $T_{\tau} = \zeta_{\tau} V_{\tau}$ surrounding~\eqref{eqn:S-T-convergence} more carefully.
After a further arbitrarily small constant rotation of $(S,U)$ together, and of $(T,V)$ together, we can apply Sard's theorem to the countably many integer coordinate lines to assume that $\pm e_y$ are regular values of $U|_{ \{ j \} \times \bR }$ for every $j \in \bZ$, and $\pm e_x$ are regular values of $V|_{\bR \times \{ k \} }$ for every $k \in \bZ$.
Only finitely many preimages occur because the fields are constant outside a large disk.
After adding finitely many corresponding preimages to $F$, every trajectory of $U$ meets the vertical integer lines transversely away from $F$, and every trajectory of $V$ meets the horizontal integer lines transversely away from $F$.
We fix a vertical strip $\cU^x_j$.
Up to the finitely many trajectories meeting $F$, every maximal connected piece of a $U$-trajectory in $\cU^x_j$ is a curve exhibiting one of the following behaviors:
\begin{enumerate}[(i)]
    \item crossing from the left side to the right side;
    \item crossing from the right side to the left side, or entering and exiting through the same side;
    \item an injective complete curve of infinite length; or
    \item a loop.
\end{enumerate}
Let $I_j$ be the set of initial points on the left side for which the integral curves are of type $(i)$, namely, curves that enter through the left side and exit through the right side.
The set $I_j$ is open, so $I_j = \bigcup_{\ell} I_{j, \ell}$ is a union of disjoint open intervals $I_{j,\ell}$.
We denote by $\cR^x_{j,\ell}$ the union of the trajectories issued from $\{ j \} \times I_{j,\ell}$, so each $\cR^x_{j,\ell}$ is a vertical flow box for $U$.
We set
\[
\cR^x_j := \bigcup_{\ell} \cR^x_{j,\ell}, \qquad \cR^x := \bigcup_{j \in \bZ} \cR^x_j
\]
so $\cR^x$ is the union of all left-to-right crossing trajectories of $U$.
Given an expression $T = \zeta V$ for $T$, we can define horizontal strips $\cR^y_{k,m}$ analogously, with $\cR^y = \bigcup_{k,m} \cR^y_{k,m}$.

We now use the Smirnov decomposition~\ref{thm:smirnov} to obtain the following estimate.
\begin{proposition}\label{prop:union-of-boxes}
In the setting above, let $U \in C^{\infty} ( \bR^2 \setminus F; \bS^1)$ be a global unit direction field and let $\eta \geq 0$ be a smooth, compactly supported function satisfying the properties of Lemma~\ref{lemma:S-vector-field}, with $S = \eta U \in C_c^{\infty}(\bR^2;\bR^2)$ and $U = e_x$ outside a compact set.
    Then,
    \[
    \int_{\bR^2 \setminus \cR^x(U)} q_x(S)_+ \leq \int_{\bR^2} |\textup{div} \, S| + \int_{\bR^2 \setminus \cR^x(U)} q_x(S)_- \, 
    \]
where $\cR^x(U)$ denotes the union of the left-to-right crossing flow boxes of $U$.
\end{proposition}
\begin{proof}
    It suffices to prove the estimate separately on each vertical strip $\cU^x_j = [j,j+1] \times \bR$.
    Up to the finitely many exceptional trajectories meeting $F$, the set $\cB_j := \cU^x_j \setminus \bigcup_{\ell} \cR^x_{j,\ell}$ is foliated by the trajectories of types $(ii)$-$(iv)$ described above.
    Let $\vartheta : \bR^2 \to [0,1]$ be a smooth function that vanishes precisely on the set $F$ to sufficiently high order and equals $1$ away from a larger compact set containing a larger neighborhood of $F$ and the compact set where $\{ U \neq e_x \}$.
    Because $U = e_x$ outside a compact set, the region $\cB_j$ is bounded: for $|y|$ sufficiently large, the whole horizontal segment $[j,j+1] \times \{y \}$ lies in the region where $U = e_x$, so every $U$-trajectory in $\cU^x_j$ is a horizontal trajectory crossing left-to-right, and therefore belongs to one of the removed flow boxes $\cR^x_{j,\ell}$.
    
    We now define, for $\delta>0$, 
    \[
    X_{\delta} := S + \delta \vartheta^2 U = ( \eta + \delta \vartheta^2) U, \qquad J_{\delta,j} := \mathbf{1}_{\cB_j} X_{\delta} \, \cL^2
    \]
    as a vector-valued measure.
    Since the region $\cB_j$ is bounded, we have
    \begin{equation}\label{eqn:X-delta}
    X_{\delta} = S + \delta \vartheta^2 U \to S \quad \text{in } \; L^1(\cB_j), \qquad \implies \qquad \int_{\cB_j} q_x(X_{\delta})_{\pm} \to \int_{\cB_j} q_x(S)_{\pm}
    \end{equation}
    because the functions $q_x(-)_{\pm}$ are Lipschitz.
    Moreover, using the explicit form of $U$ near $F$, we have
    \[
    \textup{div} \, X_{\delta} = \textup{div} \, S + \delta \, \textup{div} ( \vartheta^2 U), \qquad \text{and} \qquad \textup{div} ( \vartheta^2 U) \in L^1_{\textup{loc}}(\bR^2).
    \]
    In fact, the function $\textup{div} ( \vartheta^2 U)$ is integrable on $\Omega_j = (j, j+1) \times \bR$: near $F$, this follows from our choice of $\vartheta$, while outside a larger compact set, we have $\vartheta=1$ and $U=e_x$.
    Therefore,
    \begin{equation}\label{eqn:Omega-j-div-X-delta}
    \limsup_{\delta \downarrow 0} \int_{\Omega_j} |\textup{div} \, X_{\delta}| \leq \int_{\Omega_j} |\textup{div} \, S|.
    \end{equation}
    Since $\eta + \delta \vartheta^2>0$ on $\bR^2 \setminus F$, the polar direction of $J_{\delta,j}$ agrees with $U$ at $|J_{\delta,j}|$-a.e.~point.
    The portions $\partial \cR^x_{j,\ell}$ of the boundaries of the removed flow boxes $\cR^x_{j,\ell}$ lying in $\Omega_j$ are integral curves of $U$, so they carry no normal flux because $X_{\delta}$ is parallel to $U$.
    Exhausting $\cB_j$ by the complements of finitely many such flow boxes and passing to the limit, we obtain
    \begin{equation}\label{eqn:justify-jjdelta-normal}
        |\textup{div} \, J_{\delta,j}|(\Omega_j) \leq \int_{\Omega_j} |\textup{div} \, X_{\delta}|.
    \end{equation}
    Since $\cB_j$ is bounded and $X_{\delta}$ is bounded, the flux across the two vertical sides is finite.
    Hence, $J_{\delta,j}$ has finite distributional divergence and defines a normal $1$-current with $\mathbf{M}( \partial J_{\delta,j}) = |\textup{div} \, J_{\delta,j}|(\bR^2)$.
    Inside $\Omega_j := (j,j+1) \times \bR$, we can bound the mass $\mathbf{M} ( \partial J_{\delta,j}; \Omega_j) = |\textup{div} \, J_{\delta,j}|(\Omega_j)$  by~\eqref{eqn:justify-jjdelta-normal}.
    Applying Theorem~\ref{thm:smirnov} to $J_{\delta,j}$, we can decompose
    \[
    J_{\delta, j} = \int_{\Sigma} [\![ \gamma_{\sigma} ]\!] \, d \lambda(\sigma) + \int_{\Xi} \Gamma_{\xi} \, d \mu(\xi)
    \]
    where the first family consists of injective finite-length arcs and the second family consists of cycles, with no cancellation of mass and no boundary cancellation of the contribution from the arcs.
    Since the polar direction of $J_{\delta,j}$ agrees with $U$, the oriented tangent to $\gamma_{\sigma}$ agrees with $U$ for $\lambda$-a.e.~$\sigma$ and $\cH^1$-a.e.~point of $\gamma_{\sigma}$; the same holds for the cycle part.

    Let $\tilde{\Sigma}_j$ denote the set of $\sigma \in \Sigma$ for which $\gamma_{\sigma}$ has an endpoint in the open strip $\Omega_j = (j,j+1) \times \bR$.
    Since there is no boundary cancellation in Theorem~\ref{thm:smirnov}, we can use the~\eqref{eqn:justify-jjdelta-normal} bound 
    \[
    \lambda(\tilde{\Sigma}_j) \leq \mathbf{M} ( \partial J_{\delta,j} ; \Omega_j) \leq \int_{\Omega_j} |\textup{div} \, X_{\delta}|.
    \]
    Thus, it remains to estimate the contribution of each decomposing curve.
    For every oriented curve $\gamma : [a,b] \to \cU^x_j$, parametrized by arc-length, we define
    \[
    H(\gamma) := \int_a^b [ q_x(\dot{\gamma})_+ - q_x(\dot{\gamma})_-] \, dt = \int_a^b q_x(\dot{\gamma}) \, dt.
    \]
    We claim that, for $\lambda$-a.e.~$\sigma$ and for $\mu$-a.e.~$\xi$, it holds that
    \begin{equation}\label{eqn:curve-bounds}
    H(\gamma_{\sigma}) \leq \mathbf{1}_{\tilde{\Sigma}_j}(\sigma), \qquad H(\Gamma_{\xi}) \leq 0.
    \end{equation}
    Indeed, every cycle $\Gamma_{\xi}$ is a countable union of finite-length loops.
    Applying the second estimate of Lemma~\ref{lemma:lipschitz-curve} to each loop gives $H(\Gamma_{\xi}) \leq 0$, proving the second claim.
    
    Next, suppose that $\sigma \not\in \tilde{\Sigma}_j$, so neither endpoint of $\gamma_{\sigma}$ lies in the interior $\Omega_j$.
    Since $\gamma_{\sigma}$ is tangent to $U$, it is a finite portion of one of the $U$-trajectories contained in $\cB_j$.
    Moreover, because neither endpoint lies in $\Omega_j$, this $\gamma_{\sigma}$ forms a full trajectory segment in the strip, with its finite endpoints lying on the boundary of the strip.
    By the definition of the set $\cB_j$, this curve cannot be of type $(i)$; it also cannot be of type $(iii)$, since $\gamma_{\sigma}$ has finite length, and it cannot be of type $(iv)$, since $\gamma_{\sigma}$ is an injective arc rather than a cycle.
    Hence, $\gamma_{\sigma}$ is of type $(ii)$ and carries the orientation induced by $U$; this implies that $\la \gamma_{\sigma}(b), e_x \rg \leq \la \gamma_{\sigma}(a), e_x \rg$.
    Hence, the second bound of Lemma~\ref{lemma:lipschitz-curve} again shows that $H(\gamma_{\sigma}) \leq 0$.
    Finally, if $\sigma \in \tilde{\Sigma}_j$, then the fact that $\gamma_{\sigma}([a,b]) \subset \cU^x_j$ combined with the first bound of Lemma~\ref{lemma:lipschitz-curve} implies that $H(\gamma_{\sigma}) \leq 1$.
    This completes the proof of~\eqref{eqn:curve-bounds}.

    We now use the fact that the Smirnov decomposition of Theorem~\ref{thm:smirnov} has no mass cancellation.
    Since the functions $v \mapsto q_x(v)_{\pm}$ are positively one-homogeneous, we obtain
    \begin{align}
        \int_{\cB_j} [ q_x(X_{\delta})_+ - q_x(X_{\delta})_-] &= \int_{\Sigma} H(\gamma_{\sigma}) \, d \lambda(\sigma) + \int_{\Xi} H(\Gamma_{\xi}) \, d \mu(\xi) \notag \\
        &\leq \lambda(\tilde{\Sigma}_j) \leq \int_{\Omega_j} |\textup{div} \, X_{\delta}|, \notag \\
        \implies \int_{\cB_j} q_x (X_{\delta})_+ &\leq \int_{\cB_j} q_x(X_{\delta})_- + \int_{\Omega_j} |\textup{div} \, X_{\delta}|.\label{eqn:final-smirnov-step}
    \end{align}
    Combining the bounds~\eqref{eqn:X-delta} and~\eqref{eqn:Omega-j-div-X-delta}, we can pass to the limit in~\eqref{eqn:final-smirnov-step} to obtain
    \[
    \int_{\cB_j} q_x(S)_+ \leq \int_{\cB_j} q_x(S)_- + \int_{\Omega_j} |\textup{div} \, S|.
    \]
    Finally, summing this inequality over all $j \in \bZ$, where the strip interiors $\Omega_j$ are pairwise disjoint, we can write $\bigcup_{j \in \bZ} \cB_j = \bR^2 \setminus \cR^x(U)$ up to the vertical integer lines and the finitely many exceptional trajectories, all of which are $\cL^2$-negligible.
    This completes the proof of our assertion.
\end{proof}

We now combine the above results to prove Theorem~\ref{thm:kakeya-type-inequality}.
\begin{proof}[Proof of Theorem~\ref{thm:kakeya-type-inequality}]
We first suppose that $S,T$ are compactly supported and that $\chi \in C_c(\bR^2)$.
Let $\rho_{\ve}$ be a standard compactly supported mollifier and convolve the joint $\bR^4$-valued measure $(S,T)$ to obtain $(S_{\ve}, T_{\ve}) := (\rho_{\ve} \ast S, \rho_{\ve} \ast T)$.
Then, $\textup{div} \, S_{\ve} = \rho_{\ve} \ast (\textup{div} \, S)$ and $\textup{div}\, T_{\ve} = \rho_{\ve} \ast ( \textup{div} \, T)$, while Lemma~\ref{lemma:strict-approximation} produces the strict convergence
\[
(S_{\ve}, T_{\ve}) \to (S,T), \qquad\textup{div}\, S_{\ve} \to \textup{div} \, S, \qquad \textup{div} \, T_{\ve} \to \textup{div} \, T
\]
as $\ve \downarrow 0$.
We now fix $\ve>0$, so $S_{\ve}, T_{\ve}$ are smooth and compactly supported.
We apply Lemma~\ref{lemma:S-vector-field} separately to $S_{\ve}, T_{\ve}$, prescribing the exterior directions $e_x, e_y$ respectively.
For every sufficiently small $\tau>0$, this gives $S_{\ve,\tau} = \eta_{\ve, \tau} U_{\ve,\tau}$ and $T_{\ve,\tau} = \zeta_{\ve, \tau} V_{\ve,\tau}$, with
\[
S_{\ve,\tau} \to S_{\ve}, \qquad T_{\ve,\tau} \to T_{\ve}, \qquad \textup{div} \, S_{\ve,\tau} \to \textup{div} \, S_{\ve}, \qquad \textup{div} \, T_{\ve,\tau} \to \textup{div} \, T_{\ve}
\]
in $L^1(\bR^2)$, as $\tau \downarrow 0$.
We now fix $\tau$ and suppress $\ve,\tau$ to write $S=\eta U$ and $T = \zeta V$.
After the arbitrarily small generic perturbation described above, we may form the vertical crossing flow boxes $\cR^x = \bigcup \cR^x_j$ of $U$ and the horizontal crossing flow boxes $\cR^y = \bigcup_k \cR^y_k$ of $V$.
Within either family the boxes are pairwise disjoint up to their boundaries.

On each vertical box $\cR^x_j$, we apply the construction of Lemma~\ref{lemma:flow-box} to the non-vanishing unit field $U$ and write $U = \hat{\alpha} Z$ with $\textup{div} \, Z=0$ and $S = (\eta \hat{\alpha}) Z$.
Taking the minimum of the non-negative coefficient $\alpha := \eta \hat{\alpha}$ along each trajectory gives the decomposition $S = S_{f,j} + S_{d,j}$, as well as an associated good set $G^x_j$.
Likewise, on every horizontal box $\cR^y_k$, we decompose $T = T_{f,k} + T_{d,k}$ with good set $G^y_k$.
Because $S_{f,j}$ and $S_{d,j}$ are non-negative multiples of the same oriented vector field, the positive one-homogeneity of the functions $q_e(-)_{\pm}$ gives
\begin{equation}\label{eqn:qx-decomposition}
q_x(S)_+ = q_x(S_{f,j})_+ + q_x(S_{d,j})_+, \qquad q_x(S)_- = q_x(S_{f,j})_- + q_x(S_{d,j})_-
\end{equation}
and likewise for $T$.
On each region $\cR^x_j \cap \cR^y_k$, we have
\begin{equation}\label{eqn:qx-S-qy-T}
\min \{ q_x(S)_+, q_y(T)_+\} \leq \min \{ q_x(S_{f,j})_+ , q_y(T_{f,k})_+ \} + q_x (S_{d,j})_+ + q_y (T_{d,k})_+.
\end{equation}
We first estimate the minimum term on $G := \bigcup_{j,k} (G^x_j \cap G^y_k)$.
Using $\min \{ a,b \} \leq \sqrt{ab}$ for $a,b \geq 0$, we can apply Lemma~\ref{lemma:good-good-interaction} to each intersecting pair of boxes to obtain
\begin{align*}
    \int_{G^x_j \cap G^y_k} \chi \, \min \{ q_x (S_{f,j})_+ , q_y (T_{f,k})_+ \} \leq \| \chi \|_{L^2(G^x_j \cap G^y_k)} \Biggl[ \Bigl( \int_{\cR^x_j} |S| + |\textup{div} \, S| \Bigr) \Bigl( \int_{\cR^y_k} |T| + |\textup{div} \, T | \Bigr) \Biggr]^{\frac{1}{2}}
\end{align*}
on each $G^x_j \cap G^y_k$.
These good sets are pairwise disjoint up to null sets.
Writing $A_j := \int_{\cR^x_j} ( |S| + |\textup{div} \, S|)$ and $B_k := \int_{\cR^y_k} ( |T| + |\textup{div} \, T|)$, we may therefore apply the Cauchy-Schwarz inequality and sum over pair indices $(j,k)$ to obtain the bound
\begin{align*}
    \int_G \chi \, \min \{ q_x(S_f)_+ , q_y(T_f)_+ \} &\leq \sum_{j,k} \int_{G^x_j \cap G^y_k} \chi \min \{ q_x(S_{f,j})_+ , q_y(T_{f,k})_+ \} \\
    &\leq \sum_{j,k } \| \chi \|_{L^2(G^x_j \cap G^y_k)} (A_j B_k)^{\frac{1}{2}} \leq \Bigl( \sum_{j,k} \| \chi\|^2_{L^2(G^x_j \cap G^y_k} \Bigr)^{\frac{1}{2}} \Bigl( \sum_{j,k} A_j B_k \Bigr)^{\frac{1}{2}} \\
    &\leq \| \chi \|_{L^2(G)} \Bigl( \sum_j A_j \Bigr)^{\frac{1}{2}} \Bigl( \sum_k B_k \Bigr)^{\frac{1}{2}}.
\end{align*}
Since $\cR^x = \bigcup_j \cR^x_j$ and $\cR^y = \bigcup_k \cR^y_k$, we therefore obtain
\allowdisplaybreaks{
\begin{align*}
    &\sum_j A_j \leq \int_{\cR^x} |S| + |\textup{div} \, S|, \qquad \sum_k B_k \leq \int_{\cR^y} |T| + |\textup{div} \, T|, \\
    & \implies \int_G \chi \, \min \{ q_x(S_f)_+, q_y(T_f)_+ \} \leq \| \chi \|_{L^2} \Bigl( \int_{\cR^x} |S| + |\textup{div} \, S| \Bigr)^{\frac{1}{2}} \Bigl( \int_{\cR^y} |T| + |\textup{div} \, T | \Bigr)^{\frac{1}{2}}.
\end{align*}}
Next, we set $\cR := \cR^x \cap \cR^y$ and consider $\cR \setminus G$.
At every point of this set, at least one of the two conditions for the good sets $G^x$ and $G^y$ fails.
Thus, 
\[
\min \{ q_x(S_f)_+ , q_y(T_f)_+\} \leq \mathbf{1}_{\cR^x \setminus G^x} \, q_x(S_f)_+ + \mathbf{1}_{\cR^y \setminus G^y} \, q_y(T_f)_+.
\]
Using the bound $\chi \leq \| \chi \|_{L^{\infty}}$ and summing the first estimate in Proposition~\ref{prop:generalized-flow-boxes} over all vertical and horizontal boxes gives, for $\cR := \cR^x \cap \cR^y$,
\[
\int_{\cR \setminus G} \chi \, \min \{ q_x(S_f)_+ , q_y(T_f)_+ \} \leq \| \chi \|_{L^{\infty}} \Bigl[ \int_{\cR^x} q_x(S_f)_- + \int_{\cR^y} q_y(T_f)_- \Bigr].
\]
For the residual pieces, the second estimate of Proposition~\ref{prop:generalized-flow-boxes} gives
\begin{align*}
    \int_{\cR} \chi \, q_x(S_d)_+ &\leq \| \chi \|_{L^{\infty}} \int_{\cR^x} q_x(S_d)_+ \leq \| \chi \|_{L^{\infty}}  \int_{\cR^x} ( q_x(S_d)_- + |\textup{div}\, S|), \\
    \int_{\cR} \chi \, q_y(T_d)_+ &\leq \| \chi \|_{L^{\infty}} \int_{\cR^y} q_y(T_d)_+ \leq \| \chi \|_{L^{\infty}} \int_{\cR^y} ( q_y(T_d)_- + |\textup{div} \, T|).
\end{align*}
Combining these bounds with the decomposition~\eqref{eqn:qx-decomposition} on $\cR^x$ and its analogue for $T$ on $\cR^y$ gives 
\begin{align*}
    \int_{\cR} \chi \, \min \{ q_x(S)_+, q_y(T)_+ \} &\leq \| \chi \|_{L^2} \Bigl( \int |S| + |\textup{div} \, S| \Bigr)^{\frac{1}{2}} \Bigl( \int |T| + |\textup{div} \, T| \Bigr)^{\frac{1}{2}} \\
    & \quad + \| \chi \|_{L^{\infty}} \int_{\cR^x} (q_x(S)_- + |\textup{div} \, S|) + \| \chi \|_{L^{\infty}} \int_{\cR^y} ( q_y(T)_- + |\textup{div} \, T|).
\end{align*}
Next, we bound the contribution of the common crossing region $\bR^2 \setminus \cR = (\bR^2 \setminus \cR^x) \cup (\bR^2 \setminus \cR^y)$ as
\allowdisplaybreaks{
\begin{align*}
\int_{\bR^2 \setminus \cR} \chi \, \min \{ q_x(S)_+, q_y(T)_+ \} &\leq \| \chi \|_{L^{\infty}} \Bigl[ \int_{\bR^2 \setminus \cR^x} q_x(S)_+ + \int_{\bR^2 \setminus \cR^y} q_y(T)_+ \Bigr] \\
&\leq \| \chi\|_{L^{\infty}} \int_{\bR^2} ( |\textup{div} \, S| + |\textup{div} \, T| ) \\
& \quad+ \| \chi \|_{L^{\infty}} \Bigl[ \int_{\bR^2 \setminus \cR^x} q_x(S)_- + \int_{\bR^2 \setminus \cR^y} q_y(T)_- \Bigr].
\end{align*}}
In the second step, we bounded $\int_{\bR^2 \setminus \cR^x} q_x(S)_+$ using Proposition~\ref{prop:union-of-boxes}, and argued likewise for $\int_{\bR^2 \setminus \cR^y} q_y(T)_+$.
We then add the terms
\[
\int_{\cR^x} q_x(S)_- + \int_{\bR^2 \setminus \cR^x} q_x(S)_- = \int_{\bR^2} q_x(S)_-, \qquad \int_{\cR^y} q_y(T)_- + \int_{\bR^2 \setminus \cR^y} q_y(T)_- = \int_{\bR^2} q_y(T)_-
\]
Combining these terms proves the inequality for the approximating vector fields $S_{\ve,\tau}, T_{\ve,\tau}$, uniformly in $\tau$.
We now send $\tau \downarrow 0$.
Because the map $(v,w) \mapsto \{ \min q_x(v)_+, q_y(w)_+ \}$ is Lipschitz and one-homogeneous on $\bR^4$, the $L^1$ convergence of $S_{\ve, \tau}, T_{\ve, \tau}$ implies
\[
\min \{ q_x(S_{\ve,\tau})_+, q_y(T_{\ve,\tau})_+ \} \to \min \{ q_x(S_{\ve})_+, q_y(T_{\ve})_+ \}
\]
in $L^1$, and likewise for $q_x(S_{\ve,\tau})_-, q_y(T_{\ve,\tau})_-$, as well as their divergences, which converge in $L^1$ and in total mass.
We may therefore pass to the limit and obtain the desired estimate for the smooth fields $S_{\ve}, T_{\ve}$.
In the next step, we send $\ve \downarrow 0$, so the joint measures converge $(S_{\ve}, T_{\ve}) \to (S,T)$.
For the positively one-homogeneous functions on $\bR^4$,
\[
F(v,w) := \min \{ q_x(v)_+, q_y(w)_+ \}, \qquad F_x(v,w) := q_x(v)_-, \qquad F_y(v,w) := q_y(w)_-,
\]
Reshetnyak's continuity theorem~\ref{lemma:reshetnyak-continuity}, applied to the integrand $(x,v,w) \mapsto \chi(x) F(v,w)$, shows that
\[
\int \chi \, \min \{ q_x(S_{\ve})_+, q_y(T_{\ve})_+ \} \to \int \chi \min \{ q_{x}(S)_+, q_y(T)_+ \}, \qquad \int q_x(S_{\ve})_- \to \int q_x(S)_-,
\]
and $\int q_y(T_{\ve})_- \to \int q_y(T)_-$.
Moreover, strict convergence gives
\[
|S_{\ve}|(\bR^2) \to |S|(\bR^2), \qquad |T_{\ve}|(\bR^2) \to |T|(\bR^2), \qquad |\textup{div} \, S_{\ve}|(\bR^2) \to |\textup{div} \, S|(\bR^2)
\]
and $|\textup{div} \, T_{\ve}|(\bR^2) \to |\textup{div} \, T|(\bR^2)$.
Thus, the desired inequality again passes to the limit and proves the theorem for compactly supported vector-valued measures $S,T$ and continuous $\chi$.

Finally, when $S,T$ are arbitrary, we choose cutoffs $\eta_R$ with $0 \leq \eta_R \leq 1$ and $\eta_R \equiv 1$ on $B_R$, $\eta_R \equiv 0$ on $B_{2R}$, and $|\nabla \eta_R| \leq CR^{-1}$, where $R$ is sufficiently large so that $\textup{spt} \, \chi \subset B_R$.
For sufficiently large $R$, setting $(S^R, T^R) := (\eta_R S, \eta_R T)$ leaves the left-hand side unchanged, while
\[
\textup{div} \, S^R = \eta_R \, \textup{div} \, S + \la \nabla \eta_R , S \rg \quad \implies \quad \limsup_{R \to \infty} |\textup{div} \, S^R|(\bR^2) \leq |\textup{div} \, S|(\bR^2)
\]
and likewise for $T$, because $|\la \nabla \eta_R, S \rg|(\bR^2) \leq C R^{-1} |S|(\bR^2) \to 0$.
Moreover, the expressions $q_x(-)_{\pm}, q_y(-)_{\pm}$ converge, by positive homogeneity, using the dominated convergence theorem with respect to the corresponding variation measures $|S|,|T|$.
We may therefore send $R \to \infty$ to obtain the result for arbitrary finite $S,T$ with continuous $\chi$.

Finally, let $\chi$ be non-negative, bounded, compactly supported, and Borel, and consider the positive Radon measure $\lambda := \min \{ q_x(S)_+, q_y(T)_+ \}$.
We fix a compact set $K$ containing $\textup{spt} \, \chi$ and approximate $\chi$ by a sequence of $\chi_j \in C_c(\bR^2)$ with $0 \leq \chi_j \leq \| \chi \|_{L^{\infty}}$ such that $\chi_j \to \chi$ in $L^1( (\cL^2 + \lambda) \mres K)$.
Then, $\int \chi_j \, d \lambda \to \int \chi \, d \lambda$, so the $L^1(\cL^2)$ convergence implies
\[
\| \chi_j - \chi \|^2_{L^2} \leq 2 \, \| \chi \|_{L^{\infty}} \| \chi_j - \chi\|_{L^1} \to 0, \qquad \implies \qquad \| \chi_j \|_{L^2} \to \| \chi \|_{L^2}
\]
due to the uniform $L^{\infty}$ bound.
Moreover, $\| \chi_j \|_{L^{\infty}} \leq \| \chi \|_{L^{\infty}}$.
We may therefore apply the established estimate for each $\chi_j$ and send $j \to \infty$; this completes the proof.
\end{proof}

\section{Proof of the main theorem}\label{section:proof-surfaces}

We now use Theorem~\ref{thm:kakeya-type-inequality} to control the projected varifold measures.

\begin{proposition}\label{prop:A-div-S}
    Consider a vector-valued Radon measure $\mathbf{A} \in \cM(\bR^2 ; \bR^{2 \times 2})$ with measure-valued row-wise divergence and concentrated on a Borel set $E \subset \bR^2$ of finite Lebesgue measure.
    Then,
    \[
    \Lambda_{2,\beta} ( \mathbf{A}) \leq C |E|^{\frac{1}{2}} ( |\mathbf{A}| + |\textup{Div} \, \mathbf{A}| )  + C \, |\textup{Div} \, \mathbf{A}| + \cN_2 ( \mathbf{A}) + 2 \, \cK_2( \mathbf{A})
    \]
    with the constant $\beta = \frac{4-\pi}{8}$.
\end{proposition}
\begin{proof}
    We let $S,T$ denote the two rows of $\mathbf{A}$, which belong to $\cM( \bR^2 ; \bR^2)$, and choose a positive measure $\sigma$ dominating all entries of $\mathbf{A}$.
    Then, Lemma~\ref{lemma:2x2-matrix} gives
    \begin{align*}
    \psi_2 \Bigl( \frac{d \mathbf{A}}{d \sigma} \Bigr) &= \min \left\{ q_x \Bigl( \frac{dS}{d \sigma} \Bigr)_+, q_y \Bigl( \frac{dT}{d \sigma} \Bigr)_+ \right\} - q_x \Bigl( \frac{dS}{d \sigma} \Bigr)_- - q_y \Bigl( \dfrac{dT}{d \sigma} \Bigr)_- \, , \\
    \implies \Psi_2( \mathbf{A}) &= \int_{\bR^2} \min \{ q_x(S)_+, q_y(T)_+ \} - \int_{\bR^2} (q_x(S)_- + q_y(T)_-) .
    \end{align*}
    Since $\mathbf{A}$ is concentrated on $E$, we can apply Theorem~\ref{thm:kakeya-type-inequality} with $\chi = \mathbf{1}_{E \cap B_R}$ for a sequence of radii, so $\textup{spt} \, \chi \subset B_R$ is compact.
    In the next step, we send $R \to \infty$ and use the dominated convergence theorem, since $\cL^2(E) < \infty$, to obtain
    \[
    \Psi_2( \mathbf{A}) \leq |E|^{\frac{1}{2}} ( |S| + |\textup{div} \, S|)^{\frac{1}{2}} ( |T| + |\textup{div} \, T|)^{\frac{1}{2}} + 2 ( |\textup{div} \, S| + |\textup{div} \, T| )
    \]
    Using the arithmetic-geometric mean inequality and the equivalence of the row and matrix norms, up to a dimensional constant, we obtain
    \begin{equation}\label{eqn:Psi2(A)}
    \Psi_2( \mathbf{A}) \leq C|E|^{\frac{1}{2}} ( |\mathbf{A}| + |\textup{Div} \, \mathbf{A}|) + C \, |\textup{Div} \, \mathbf{A}|.
    \end{equation}
    For $R \in \textup{SO}(2)$, we consider the rotated measure $\mathbf{A}^R := R^t (R^t)_{\#} \mathbf{A} R$.
    We combine the inequality~\eqref{eqn:Psi2(A)} with the bound of Lemma~\ref{lemma:signed-average}, for $\beta = \frac{4-\pi}{8}>0$, to obtain
    \begin{align*}
    \frac{4}{\pi} \Lambda_{2,\beta} ( \mathbf{A}) &\leq \int_{\textup{SO}(2)} \Psi_2( \mathbf{A}^R) \, dR + \cN_2( \mathbf{A}) + 2 \, \cK_2( \mathbf{A}) \\
    &\leq C |E|^{\frac{1}{2}} ( |\mathbf{A}| + |\textup{Div} \, \mathbf{A}|) + C \, |\textup{Div} \, \mathbf{A}| + \cN_2( \mathbf{A}) + 2 \, \cK_2( \mathbf{A})
    \end{align*}
    because the total variation and divergence measures are preserved under rotation.
\end{proof}

On the other hand, Lemma~\ref{lemma:averaged-coercivity} allows us to bound the Haar average $\int \Lambda_{m,\beta}( \mathbf{A}_P) \, dP$ by a dimensional constant whenever $\sup_{T \in \bG(N,m)} \| B_{\Psi}(T) - T \| \leq \ve$.
We now prove that this constant is positive in every ambient dimension $N$ for $m=2$.
\begin{lemma}\label{lemma:key-m=2}
    For every $N \geq 3$ and $\beta = \frac{4-\pi}{8}$, we have $a_{N,2} > \frac{4 \beta}{N}$ where $a_{N,2}$ is defined in~\eqref{eqn:d-N-m-beta}.
\end{lemma}
\begin{proof}
    Let $T \in \bG(N,2)$ and let $(u,v)$ be a Haar-random orthonormal basis of a 2-frame in $\bR^N$, so $P := u \otimes u + v\otimes v$ is a Haar-random element of $\bG(N,2)$.
    For a fixed $T \in \bG(N,2)$, we consider its projection to $P$ computed by $A := \begin{pmatrix}
        \la Tu, u\rg & \la Tu, v\rg \\
        \la Tv, u\rg & \la Tv, v\rg \\
    \end{pmatrix}$ with eigenvalues $\lambda_1 \geq \lambda_2$, so $a_{N,2} = \bE[\lambda_2]$.
    Since the Haar measure is $O(N)$-invariant for all $R \in O(N)$, we know that $\mr{tr}(P) = 2$, so
    \begin{equation}\label{eqn:El1+l2}
    \bE[\lambda_1 + \lambda_2] = \bE[\mr{tr}(TP)] = \mr{tr}(T\tfrac{2I}{N}) = \frac{4}{N}.
    \end{equation}
    Recalling the general property of symmetric matrices used in Lemma~\ref{lemma:symmetric-positive-definite-average}, we have
    \[
    (R_\theta^tAR_\theta)_{11} - (R_\theta^tAR_\theta)_{22} = (\lambda_1-\lambda_2)\cos(2\theta-\theta_0).
    \]
    Because $\frac{1}{2\pi}\int_0^{2\pi}|\cos(2\theta-\theta_0)|\,d\theta= \frac{2}{\pi}$, this implies that
    \[
    \lambda_1-\lambda_2 = \frac{1}{4}\int_0^{2\pi} | (R_\theta^tAR_\theta)_{11} - (R_\theta^tAR_\theta)_{22} |\,d\theta.
    \]
    Integrating this over all $A$ shows that 
    \begin{equation}\label{eqn:El1-l2pi/2}
    \bE[\lambda_1 - \lambda_2] =\frac{\pi}{2}\bE[|\la Tu,u\rg - \la Tv, v\rg|].
    \end{equation}
    For $(x,y) := \frac{1}{\sqrt{2}}(u+v, u-v)$, the pair $(x,y)$ is equally distributed and $\la Tu,u\rg - \la Tv,v\rg = 2\la Tx,y\rg$.
    For any fixed $x$, the complementary variable $y$ is equally distributed on $x^\perp \cap \bS^{N-1} \cong \bS^{N-2}$.
    For $w \in \bS^{N-2} \subset x^\perp$, we have $\bE[|\la w, y\rg| | x] = c_N |w|$ by rotational invariance.
    Writing $\omega_{k} := \cH^{k}(\bS^{k}) = \frac{2\pi^{\frac{k+1}{2}}}{\Gamma(\frac{k+1}{2})}$ the surface area, we can compute
    \[
    c_N = \frac{1}{\omega_{N-2}}\int_{z \in \bS^{N-2}}|z_1|\,d\cH^{N-2}(z) = \frac{2\omega_{N-3}}{\omega_{N-2}}\int_0^1 t(1-t^2)^{\frac{N-4}{2}}\,dt  = \frac{\Gamma(\frac{N-1}{2})}{\sqrt{\pi}\Gamma(\frac{N}{2})}
    \]
    integrating over the two copies of spheres $\bS^{N-3}$ cut out by $|z_1| = t$ which have radius $\sqrt{1-t^2}$.

    Let $s := \la Tx, x\rg$ and since $x \perp y$.
    We have $\la Tx, y\rg =\la Tx - sx, y\rg$ so $|Tx - sx|^2 =|Tx|^2 - s^2 = s(1-s)$ and therefore $\bE[|\la Tx, y \rg|\, | x] = c_N\sqrt{s(1-s)}$.
    The distribution for $s$ is given by $\frac{N-2}{2}(1-s)^\frac{N-4}{2} \mathbf{1}_{(0,1)}$.
    Indeed, for $x = (\cos(\varphi)\omega, \sin(\varphi)\eta)$ with $(\omega, \eta , \varphi) \in \bS^1 \times \bS^{N-3} \times (0, \frac{\pi}{2})$, we have $s = \cos^2(\varphi)$.
    In spherical coordinates, we can write
    \[
    d\cH^{N-1} = \cos(\varphi)(\sin(\varphi))^{N-3}\,d\varphi d\cH^{1}(\omega)d\cH^{N-3}(\eta).
    \]
    Therefore, $ds = -2\cos(\varphi)\sin(\varphi)\,d\varphi$ and the induced density on $s$ is proportional to $(1-s)^{\frac{N-4}{2}}$, and normalizing this to a probability measure gives the claimed distribution for $s$ above.

    Now, we can compute
    \[
    \bE [|\la Tu,u\rg - \la Tv,v\rg|] = (N-2)c_N \int_0^1 \sqrt{s}(1-s)^{\frac{N-3}{2}}\,ds  = \frac{\Gamma(\frac{N-1}{2})^2}{\Gamma(\frac{N-2}{2})\Gamma(\frac{N+2}{2})}
    \]
    and the Cauchy-Schwarz inequality shows 
    \[
\Gamma\Bigl(\frac{N-1}{2}\Bigr)^2 < \Gamma\Bigl(\frac{N-2}{2}\Bigr)\Gamma\Bigl(\frac{N}{2}\Bigr) \qquad \implies \qquad \bE[|\la Tu,u\rg - \la Tv,v\rg|] < \frac{\Gamma\bigl(\frac{N}{2}\bigr)}{\Gamma\bigl(\frac{N+2}{2}\bigr)} = \frac{2}{N}.
    \]
    This implies that $\bE[\lambda_1 - \lambda_2]< \frac{\pi}{N}$ from equation~\eqref{eqn:El1-l2pi/2}.
    Finally, combining this with~\eqref{eqn:El1+l2}, we have 
    \[
    a_{N,2} := \bE[\lambda_2] = \tfrac{1}{2}(\bE[\lambda_1+\lambda_2] - \bE[\lambda_1 - \lambda_2]) > \tfrac{1}{2}\bigl(\tfrac{4}{N} - \tfrac{\pi}{N} \bigr) = \tfrac{4-\pi}{2N} 
    \]
    as claimed.
    This proves our assertion.
\end{proof}
\begin{proof}[Proof of Theorem~\ref{thm:anisotropic-MS}]
Let $\ve := \sup_{T \in \bG(N,2)} \| B_{\Psi}(T) - T \|$.
Given a $2$-plane $P \in \bG(N,2)$, we consider the projected measure $\mathbf{A}_P$ introduced in Section~\ref{section:smallest-eigenvalue} for the varifold $V = v(M,\theta)$, which satisfies the disintegration~\eqref{eqn:density-AP-muP}.
Because compression by $E_P$ is $1$-Lipschitz, we obtain
\begin{equation}\label{eqn:projection-functionals}
\left\| \frac{d \mathbf{A}_P}{d \mu_P}(y) - Q_P(y) \right\| \leq \ve, \qquad \text{where} \quad Q_P(y) := \int_{\tilde{\pi}^{-1}_P(y)} E^t_P T E_P \, d \sigma_{P,y} .
\end{equation}
The expression $Q_P(y)$ is positive semidefinite, because each $E^t_P T E_P$ is, so the functions $A \mapsto \lambda_2(A)_-$ and $A \mapsto \| \textup{skew} \, A\|_{\textup{op}}$ vanish on it.
Using~\eqref{eqn:projection-functionals} and integrating with respect to $\mu_P$ and using $\mu_P (\bR^2) = \|V\|(\bR^N)$, we obtain, for every $P \in \bG(N,2)$,
\[
\cN_2( \mathbf{A}_P) \leq \ve \|V\|(\bR^N) \qquad \text{and} \qquad \cK_2(\mathbf{A}_P) \leq \ve \|V\|(\bR^N).
\]
We apply Proposition~\ref{prop:A-div-S} to each of the projected measures $\mathbf{A}_P$ and use the bounds of Lemma~\ref{lemma:A-P-projection-properties}, where $\mathbf{A}_P$ is concentrated on a set $G_P$ with $\cL^2(G_P) \leq \cH^2(M)$.
This gives
\begin{equation}\label{eqn:combine-two-propositions}
    \Lambda_{2,\beta}( \mathbf{A}_P) \leq C \, \cH^2(M)^{\frac{1}{2}} \bigl[ \, \|V\|(\bR^N) + |\delta_{\Psi} V| (\bR^N) \bigr] + C \, |\delta_{\Psi} V| (\bR^N) + C \ve \|V\|(\bR^N)
\end{equation}
for $\beta = \frac{4-\pi}{8}$.
We now average over $P \in \bG(N,2)$, using Lemma~\ref{lemma:averaged-coercivity} and~\eqref{eqn:combine-two-propositions}, to obtain
\begin{align*}
    & (a_{N,2} - 4 N^{-1}\beta - (1+2\beta) \ve ) \, \|V\|(\bR^N) \\
    &\leq \int_{\bG(N,2)} \Lambda_{2,\beta}(\mathbf{A}_P) \, dP  \\
    &\leq C_1 \, \cH^2(M)^{\frac{1}{2}} \bigl[ \, \|V\|(\bR^N) + |\delta_{\Psi} V| (\bR^N) \bigr] + C_2 \, |\delta_{\Psi} V| (\bR^N) + C_3 \ve \|V\|(\bR^N).
\end{align*}
Using Lemma~\ref{lemma:key-m=2}, we have $a_{N,2} - \frac{4\beta}{N}>0$, so we can take $\ve_N = \frac{1}{2 (1+C_3+2\beta)} ( a_{N,2} - \frac{4 \beta}{N}) > 0$ to make the coefficient on the left-hand side 
\[
a_{N,2} - 4 N^{-1}\beta - (1+2\beta) \ve - C_3 \ve > \tfrac{1}{2} (1+C_3+2\beta) \ve_N > \tfrac{1}{2} \ve_N
\]
for $\ve \in (0,\ve_N)$.
Rearranging terms by $\ve_N^{-1}$ and writing $C = C(\Psi, N)$, we arrive at
\begin{equation}\label{eqn:bound-before-rescaling}
\|V \|(\bR^N) \leq C \, \cH^2(M)^{\frac{1}{2}} \, \bigl[ \|V\|(\bR^N) + |\delta_{\Psi} V|(\bR^N) \bigr] + C \, |\delta_{\Psi} V|(\bR^N).
\end{equation}
Next, we dilate $V$ by the map $x \mapsto \rho x$ and let $V_{\rho} := (x \mapsto \rho x)_{\#}V$, which satisfies
\[
\|V_{\rho} \|(\bR^N) =\rho^2 \|V\|(\bR^N), \qquad |\delta_{\Psi} V_{\rho}|(\bR^N) = \rho |\delta_{\Psi} V|(\bR^N), \qquad \cH^2(M_{\rho}) = \rho^2 \cH^2(M).
\]
We choose the scale $\rho$ to satisfy $\rho = \frac{1}{10} C^{-1} \cH^2(M)^{-\frac{1}{2}}$.
With this choice of $\rho$, we see that
\begin{align*}
    \|V_{\rho}\|(\bR^N) - C \cH^2(M_{\rho})^{\frac{1}{2}} \|V_{\rho} \|(\bR^N) &= \rho^2 \|V\|(\bR^N) \, \bigl[ 1 - C \rho \, \cH^2(M)^{\frac{1}{2}} \bigr] > \tfrac{1}{2} \rho^2 \|V\|(\bR^N).
\end{align*}
Consequently, we can apply the bound~\eqref{eqn:bound-before-rescaling} to the varifold $V_{\rho}$ and obtain
\begin{align*}
\|V\|(\bR^N) \leq \tilde{C}(N,\Psi) \, \cH^2(M)^{\frac{1}{2}}  \,|\delta_{\Psi}V|(\bR^N).
\end{align*}
We note that, under the bound $\sup_T \|B_{\Psi}(T) - T\| \leq \ve_N$, we automatically have $\sup \|B_{\Psi}(T) \| \leq 1+\ve_N$, so every dependence of the above bounds on the anisotropy is controlled uniformly in terms of $N$.
In particular, if $\Theta^2(\|V\|,-) \geq \theta_0 >0$ holds at $\|V\|$-a.e.~point, then $\cH^2(M) \leq \theta_0^{-1} \|V\|(\bR^N)$.
Using this inequality above completes the proof of the theorem.
\end{proof}

\bibliography{ref}

@article{derosa-ghiraldin,
    author = {De Philippis, G. and De Rosa, A. and Ghiraldin, F.},
    title = {{Rectifiability of varifolds with locally bounded first variation with respect to anisotropic surface energies}},
    journal = {Comm. Pure Appl. Math.},
    year = {2018},
    volume = {71},
    pages = {1123-1148},
}

@article{anisotropic-min-max,
    author = {De Philippis, G. and De Rosa, A.},
    title = {{The anisotropic min-max theory: Existence of anisotropic minimal and CMC surfaces}},
    journal = {Comm. Pure Appl. Math.},
    year = {2023},
    volume = {77},
    number = {7},
    pages = {3184-3226},
}

@article{derosa-kolasinski,
    author = {De Rosa, A. and Kolasi\'nski, S.},
    year = {2020},
    title = {{Equivalence of the Ellipticity Conditions for Geometric Variational Problems}},
    journal = {Comm. Pure Appl. Math.},
    volume = {73},
    pages = {2473-2515},
    url = {https://doi.org/10.1002/cpa.21890},
}

@article{simon-gmt,
    author = {Simon, L.},
title = {{Lectures on Geometric Measure Theory}},
year = {1984},
journal = {Proc. Centre Math. Appl.}, 
volume = {3},
}

@article{derosa-tione-regularity,
    author = {De Rosa, A. and Tione, R.},
    title = {{Regularity for graphs with bounded anisotropic mean curvature}},
    journal = {Invent. math.},
    volume = {230},
    pages = {463-507},
    year = {2022},
    url = {https://doi.org/10.1007/s00222-022-01129-6},
}

@article{allard-first-variation,
    author = {Allard, W. K.},
    title = {{On the first variation of a varifold}},
    journal = {Ann. of Math. (2)},
    volume = {95},
    year = {1972},
    pages = {417-491},
    doi = {10.2307/1970868},
}

@article{anisotropic-michael-simon,
    author = {De Philippis, G. and Pigati, A.},
    title = {{Michael-Simon inequality for anisotropic energies close to the area via multilinear Kakeya-type bounds}},
    journal = {arXiv:2601.10647},
    year = {2026},
}

@article{smirnov,
    author = {Smirnov, S.K.},
    title = {{Decomposition of solenoidal vector charges into elementary solenoids and the structure of normal one-dimensional currents}},
    journal = {St. Petersburg Math. J.},
    volume = {5},
    year = {1994},
    number = {4},
    pages = {841-867},
}

@article{reshetnyak,
    author = {Spector, D.},
    title = {{Simple proofs of some results of Reshetnyak}},
    journal = {Proc. Amer. Math. Soc.},
    volume = {139},
    pages = {1681-1690},
    year = {2011}, 
}

@book{ambrosio-fusco-pallara,
    author = {Ambrosio, L. and Fusco, N. and Pallara, D.},
    title = {{Functions of bounded variations and free discontinuity problems}},
    series = {Oxford Mathematical Monographs},
    year = {2000},
    publisher = {The Clarendon Press, Oxford University Press},
}

@inproceedings{allard-regularity,
    author = {Allard, W. K.},
    title = {{An integrality theorem and a regularity theorem for surfaces whose first variation with respect to a parametric elliptic integrand is controlled}},
    booktitle = {Geometric measure theory and the calculus of variations},
    series = {Proc. Sympos. Pure Math.},
    publisher = {American Mathematical Society},
    volume = {44},
    year = {1986},
}

@book{colding-minicozzi,
  title={{A Course in Minimal Surfaces}},
  author={Colding, T. H. and Minicozzi, W. P., II},
  series={{Graduate Studies in Mathematics}},
  volume={121},
  year={2011},
  publisher={{American Mathematical Society}},
  address={Providence, RI},
  isbn={978-0-8218-5323-8}
}

@article{michael-simon,
    author = {Michael, J.H. and Simon, L.M.},
    title = {{Sobolev and mean-value inequalities on generalised submanifolds of $\mathbb{R}^n$}},
    journal = {Comm. Pure Appl. Math.},
    volume = {26},
    year = {1973},
    pages = {361-379},
}

@article{brendle-sharp-isoperimetric,
    author = {Brendle, S.},
    title = {{The isoperimetric inequality for a minimal submanifold in Euclidean space}},
    journal = {J. Amer. Math. Soc.},
    volume = {34},
    year = {2021},
    pages = {595-603},
}

@article{brendle-eichmair,
    author = {Brendle, S. and Eichmair, M.},
    title = {{Proof of the Michael–Simon–Sobolev inequality using optimal transport}},
    journal = {J. reine angew. Math.},
    volume = {2023},
    number = {804}, 
    year = {2023},
    pages = {1-10},
}

@article{gmt-differential-inclusions,
    author = {De Lellis, C. and De Philippis, G. and Kirchheim, B. and Tione, R.},
    title = {{Geometric measure theory and differential inclusions}},
    journal = {Ann. Fac. Sci. Toulouse Math.},
    year = {2021},
    volume = {30},
    pages = {899-960},
}

@article{vanishing-mass-conjecture,
    author = {Gennaioili, L. and Rindler, F.},
    title = {{Concentration phenomena and the Vanishing Mass Conjecture}},
    journal = {arXiv:2608.20899},
    year = {2026},
}

@article{near-area-functional,
    author = {Firester, B. and Tsiamis, R.},
    title = {{The anisotropic Michael-Simon inequality}},
    journal = {Preprint},
    year = {2026},
}

\end{document}